\documentclass[a4paper,11pt]{amsart}
\usepackage{amssymb,amsfonts,amsmath}

\def\d{\delta}

\def\C{\mathbb{C}}
\def\c2{\mathbb{C}^2}
\def\R{\mathbb{R}}

\def\N{\mathbb{N}}
\def\P{\mathbb{P}}

\def\1{\mathbf{1}}

\def\a{\alpha}
\def\b{\beta}
\def\e{\varepsilon}
\def\l{\lambda}
\def\f{\varphi}

\def\p{\psi}
\def\r{\varrho}
\def\om{\omega}

\newtheorem{lem}{Lemma}[section]
\newtheorem{pro}[lem]{Proposition}
\newtheorem{defi}[lem]{Definition}
\newtheorem{def/not}[lem]{Definition/Notations}

\newtheorem{thm}[lem]{Theorem}
\newtheorem{ques}[lem]{Question}
\newtheorem{cor}[lem]{Corollary}
\newtheorem{rqe}[lem]{Remark}

\newtheorem{exa}[lem]{Example}
\newtheorem{exas}[lem]{Examples}

\begin{document}

\title[Definition of the complex Monge-Amp\`ere operator]
{Domains of definition of Monge-Amp\`ere operators on compact
K\"ahler manifolds}

\author{DAN COMAN, VINCENT GUEDJ, and AHMED ZERIAHI}
\thanks{Dan Coman was partially supported by the NSF Grant DMS
0500563.} 
\address{Dan Coman: Department of Mathematics, 215 Carnegie Hall,
Syracuse University, Syracuse, NY 13244-1150, USA. dcoman@syr.edu}
\address{Vincent Guedj and Ahmed Zeriahi: Laboratoire Emile
Picard, UMR 5580, Universit\'e Paul Sabatier, 118 route de
Narbonne, 31062 TOULOUSE Cedex 04, FRANCE.
guedj@picard.ups-tlse.fr, zeriahi@picard.ups-tlse.fr}

\begin{abstract} Let $(X,\omega)$ be a compact K\"ahler manifold.
We introduce and study the largest set $DMA(X,\omega)$ of
$\omega$-plurisubharmonic (psh) functions on which the complex
Monge-Amp\`ere operator is well defined. It is much larger than
the corresponding local domain of definition, though still a
proper subset of the set $PSH(X,\om)$ of all $\om$-psh functions.

We prove that certain twisted Monge-Amp\`ere operators are well
defined for all $\omega$-psh functions. As a consequence, any
$\om$-psh function with slightly attenuated singularities has
finite weighted Monge-Amp\`ere energy.
\end{abstract}

\maketitle

\section*{Introduction}

\par It is well known that the complex Monge-Amp\`ere operator is
not well defined for arbitrary plurisubharmonic (psh) functions.
Bedford and Taylor \cite{BT82} found a way to define it for
locally bounded psh functions. Later the definition was extended
to classes of unbounded functions (see \cite{Si85}, \cite{De93},
\cite{FS95}). Whenever defined, the Monge-Amp\`ere operator was
shown to be continuous along decreasing sequences, but it is
discontinuous along sequences in $L^p$. The natural domain of
definition of the Monge-Amp\`ere operator on open sets in ${\Bbb
C}^n$ was recently characterized in \cite{Ce04}, \cite{Bl04},
\cite{Bl06}.

\par We consider here the problem of defining Monge-Amp\`ere
operators on a compact K\"ahler manifold $X$ of complex dimension
$n$. Let $PSH(X,\omega)$ denote the set of
$\omega$-plurisubharmonic ($\omega$-psh) functions on $X$. Here,
and throughout the paper, $\omega$ is a fixed K\"ahler form on
$X$. Recall that an upper semicontinuous function $\f\in L^1(X)$
is called $\omega$-psh if the current $\omega_\f:=\omega+dd^c\f$
is positive. Motivated by the results of \cite{BT82}, \cite{Ce04},
\cite{Bl04}, \cite{Bl06}, it is natural to define the domain of
the Monge-Amp\`ere operator as follows: \vskip.2cm

\noindent {\bf Definition.} {\it Let $DMA(X,\omega)$ be the set of
functions $\varphi\in PSH(X,\omega)$ for which there is a positive
Radon measure $MA(\varphi)$ with the following property: If
$\{\varphi_j\}$ is any sequence of bounded $\omega$-psh functions
decreasing to $\varphi$ then $(\om+dd^c {\f_j})^n\rightarrow
MA(\varphi)$, in the weak sense of measures. We set
$$
\omega_\f^n=(\om+dd^c\f)^n:=MA(\f).
$$
}

\par According to this definition, $DMA(X,\omega)$ is the largest
set of $\omega$-psh functions on which the Monge-Amp\`ere operator
$(\om+dd^c\cdot)^n$ can be defined so that it is continuous with
respect to decreasing sequences of bounded $\omega$-psh functions.
It includes all the classes in which the operator was previously
defined, either as a consequence of the local theory, or genuinely
in the compact setting (the class ${\mathcal E}(X,\omega)$ from
\cite{GZ2}). The set $DMA(X,\omega)$ is a proper subset of
$PSH(X,\omega)$ (see Examples \ref{E:curve} and \ref{E:Evans}). We
show in Proposition \ref{P:dec} that the operator is continuous
under decreasing sequences in its domain. Moreover, if $\f\in
DMA(X,\omega)$ then we prove in Proposition \ref{P:cLn} that the
set of points where $\f$ has positive Lelong number is at most
countable.

There are several properties of $DMA(X,\om)$ that we expect to
hold. We discuss these in section \ref{SS:subclasses}, and we
introduce a few subclasses of interest, especially the class
$\widehat{DMA}(X,\om)$ (Definition \ref{D:DMAhat}): a function $\f
\in PSH(X,\om)$ belongs to this class if it is in $DMA(X,\om)$ and
moreover for any sequence of bounded $\om$-psh function $\f_j$
decreasing towards $\f$, and for any ``test function'' $u \in
PSH(X,\om) \cap L^{\infty}(X)$,
$$
\int_X u (\om+dd^c \f_j)^n \longrightarrow \int_X u (\om+dd^c \f)^n.
$$
This convergence property interpolates inbetween the two natures
of the Monge-Amp\`ere measure $(\om+dd^c \f)^n$: on one hand it is
stronger than the weak convergence in the sense of positive Radon
measures (any smooth test function is $C\om$-psh for some constant
$C>0$), on the other hand it is weaker than the convergence in the
sense of Borel measures. We show (Theorem \ref{T:comparison}) that
a generalized comparison principle holds in
$\widehat{DMA}(X,\om)$, and that all concrete classes under
consideration are subsets of this class (see Corollary
\ref{C:EDMAhat} and Theorem \ref{T:DMAloc}).

\vskip.1cm

\par In \cite{GZ2} a class ${\mathcal E}(X,\omega)\subset
PSH(X,\omega)$ was introduced, on which the Monge-Amp\`ere
operator is well defined and continuous along decreasing
sequences, hence ${\mathcal E}(X,\omega)\subset DMA(X,\omega)$.
Defining this class requires that one works globally on a compact
manifold, hence many of its properties have no analogue in the
local context (see \cite{GZ2}). We study in Section \ref{S:E} more
general classes ${\mathcal E}(T,\om)$ of $\om$-psh functions with
finite energy with respect to a closed positive current $T$. These
help us in studying twisted Monge-Amp\`ere operators, which happen
to be well defined in all of $PSH(X,\om)$ (Theorems \ref{T:tma1}
and \ref{T:tma2}). As a consequence, we show that any $\om$-psh
function with slightly attenuated singularities has finite energy
(Corollary \ref{C:atten}):

\vskip.2cm

\noindent {\bf Theorem.} {\it Let $\chi:\R^- \rightarrow \R^-$ be
a smooth convex increasing function, with $\chi'(-1) \leq 1$,
$\chi'(-\infty)=0$. Fix $\f \in PSH(X,\om)$ with $\sup_X \f \leq
-1$. Then
$$
\chi \circ \f \in {\mathcal E}(X,\om) \subset DMA(X,\om),
$$
hence its Monge-Amp\`ere
measure does not charge pluripolar sets.
}
\vskip.2cm

Note that it is necessary to slightly attenuate the singularities
of $\f$ (condition $\chi'(-\infty)=0$), since functions in
${\mathcal E}(X,\om)$ have zero Lelong numbers at all points (see
Lemma \ref{L:ELel}).

\par In Sections \ref{S:local} and \ref{S:Sob} we consider the set
$DMA_{loc}(X,\omega)$ on which the Monge-Amp\`ere operator is
defined as a consequence of the local theory (\cite{Bl04},
\cite{Bl06}). We prove in Theorem \ref{T:DMAloc} that this local
domain can be characterized in terms of energy classes. Moreover,
$DMA_{loc}(X,\om)$ is a proper subset of $DMA(X,\om)$ and consists
of functions whose gradient is square integrable (Proposition
\ref{P:WnW12}). In Proposition \ref{P:div} we show that
$\omega$-psh functions, bounded in a neighborhood of an ample
divisor, belong to $DMA_{loc}(X,\omega)$. In Proposition
\ref{P:wgrad} we prove that the measure
$\chi''(\varphi)d\varphi\wedge d^c \varphi \wedge\omega^{n - 1}$
has density in $L^1(X)$ for every $\f\in PSH(X,\omega)$, where
$\chi$ is any convex increasing function. It is an interesting
problem to study the stability of subclasses in $DMA(X,\omega)$
under standard geometric constructions. We show that
$DMA_{loc}(X,\omega)$ is not preserved by blowing up, but behaves
well under blowing down.

\vskip.1cm

\par We conclude this paper by analyzing further the connection between
$DMA(X,\om)$ and various energy classes in some concrete cases of
K\"ahler surfaces. The motivation comes from the fact that energy
classes have important properties, such as convexity and stability
under taking maximum. Hence an equivalent description of
$DMA(X,\om)$ in terms of such energy classes would be very useful.
Let ${\mathcal E}(\om,\om)$ be the class of $\om$-psh functions on
$X$ such that the trace measure $\om_{\f} \wedge \om$ does not
charge the set $\{\f=-\infty\}$. Then both $DMA_{loc}(X,\om)$ and
${\mathcal E}(X,\om)$ are contained in ${\mathcal E}(\om,\om)$. We
give some evidence that ${\mathcal E}(\om,\om)$ might be equal to
$DMA(X,\om)$ when $\dim_{\C} X=2$.

\vskip.2cm

\noindent {\bf Acknowledgement.} D. Coman is grateful to the
Laboratoire Emile Picard for the support and hospitality during
his visits in Toulouse.

\section{General properties and examples}\label{S:general}

\subsection{Lelong numbers and other constraints}\label{SS:Lelong}

It is important to
require in the definition that the convergence holds on {\it any}
sequence of bounded $\om$-psh functions
(see Example \ref{E:curve}). This allows us to show that
the operator is continuous under decreasing sequences in
$DMA(X,\om)$. We have in fact the following more general property.

\begin{pro}\label{P:dec} Let $\epsilon_j\geq0$,
$\epsilon_j\rightarrow0$, and let $\f_j \in
DMA(X,(1+\epsilon_j)\om)$ be a decreasing sequence towards $\f \in
DMA(X,\om)$. Then
$$
((1+\epsilon_j)\om+dd^c \f_j)^n \longrightarrow (\om+dd^c \f)^n
$$
in the weak sense of measures.
\end{pro}

\begin{proof} We first assume that all $\epsilon_j=0$.
Fix $\chi$ a test function and set, for an integer $k>0$,
$\f_j^k:=\max(\f_j,-k) \in PSH(X,\om) \cap L^{\infty}(X)$. Since
$\f_j \in DMA(X,\om)$, we can find an increasing sequence $k_j$ so
that
$$\left|\langle(\om+dd^c \f_j^{k_j})^n,\chi\rangle-\langle
(\om+dd^c \f_j)^n, \chi \rangle \right| \leq 2^{-j}.$$ Thus
$\tilde{\f_j}:=\max(\f_j,-k_j)$ is a sequence of bounded $\om$-psh
functions decreasing towards $\f$, hence $\langle
\om_{\tilde{\f_j}}^n,\chi \rangle \rightarrow \langle \om_\f^n,
\chi \rangle$. The desired convergence follows.

\par In the general case, subtracting a constant we may assume that
$\varphi_1<0$. The sequence of measures
$((1+\epsilon_j)\om+dd^c\f_j)^n$ has bounded mass, so by passing
to a subsequence we may assume that it converges weakly to a
measure $\mu$. By taking another subsequence we may assume that
$\epsilon_j$ decreases to 0. Then
$\varphi'_j=\varphi_j/(1+\epsilon_j)\in DMA(X,\omega)$ is a
decreasing sequence to $\varphi$, so
$\omega_{\varphi'_j}^n\rightarrow\omega_\varphi^n$. We conclude
that $\mu=\omega_\varphi^n$.
\end{proof}

\par If $\varphi\in PSH(X,\omega)$, let $\nu(\f,x)$ be the
Lelong number of $\varphi$ at $x\in X$, and
$$E_{\e}(\f)=\{x\in X\,/\,
\nu(\f,x)\geq\e\}\;,\;E^+(\f)=\bigcup_{\e>0}E_\e(\varphi).$$

\begin{pro}\label{P:cLn}
If $\f \in DMA(X,\om)$ then
$\om_\f^n(\{p\})\geq\nu^n(\f,p)$, for all $p\in X$. Moreover, the
set $E^+(\f)$ is at most countable and
$$\sum_{x \in E^+(\f)}\nu^n(\f,x)\leq \int_X \om^n.$$
\end{pro}

\begin{proof} Observe that the measure $\om_\f^n$ has at
most countably many atoms. We are going to show that if
$\nu(\f,x)>0$ then $\om_\f^n$ has an atom at $x$. Let
$g=\chi\,\log dist(\cdot, x)$, where $\chi\geq0$ is a smooth cut
off function such that $\chi=1$ in a neighborhood of $x$. Since
$\omega$ is K\"ahler, $dd^c g \geq 0$ near $x$, and $g$ is smooth
on $X\setminus\{x\}$, we can find $\e_0>0$ such that $\e_0 g \in
PSH(X,\om)$. Consider
$$\f_j:=\max(\f,\e_0 g-j)\in PSH(X,\om)
\cap L^{\infty}_{loc}(X \setminus \{x\}).$$ Proposition
\ref{P:div} from Section \ref{S:Sob} implies $\f_j \in
DMA(X,\om)$. Since $\f_j$ is in the local domain of definition of
the Monge-Amp\`ere operator, we have by \cite{De93}
$$\om_{\f_j}^n(\{x\})\geq\nu^n(\f_j,x)=(\min
(\e_0,\nu(\f,x)))^n.$$ Using Proposition \ref{P:dec}, we infer
that $\om_\f^n$ has a Dirac mass at $x$.

\par By \cite{De92}, there exist sequences $\e_j,\d_j \searrow 0$,
and $\psi_j\in PSH(X,(1+\d_j)\om)$, $\psi_j\searrow\varphi$, such
that $\psi_j$ is smooth in $X \setminus E_{\e_j}(\f)$ and $\sup_{x
\in X} |\nu(\psi_j,x)-\nu(\f,x)| \rightarrow 0$. Using Siu's
theorem \cite{Siu74} and the previous discussion, we conclude that
$E_{\e_j}(\f)$ is a finite set and
$$(1+\delta_j)^n\int_X\om^n\geq\sum_{x\in
E^+(\f)}\nu^n(\psi_j,x),\;
((1+\delta_j)\om+dd^c\psi_j)^n(\{p\})\geq\nu^n(\psi_j,p),$$ where
$p\in X$. We conclude by Proposition \ref{P:dec}, letting
$j\to+\infty$.
\end{proof}

\par Proposition \ref{P:cLn} provides examples of
functions not in $DMA(X,\om)$.

\begin{exa}\label{E:curve} Let $X={\Bbb P}^2$, $\om$ be the
Fubini-Study K\"ahler form, and $\f \in PSH(\P^2,\om)$ be so that
$\om_\f=d^{-1}[{\mathcal C}]$, where $[\mathcal C]$ is the current
of integration along an algebraic curve $\mathcal C$ of degree $d
\geq 1$. By Proposition \ref{P:cLn}, $\f \notin DMA({\Bbb
P}^2,\om)$ since $E^+(\varphi)$ is not countable. Alternatively,
we can construct two sequences of functions in $DMA(\P^2,\om)$
decreasing to $\varphi$ with constant Monge-Amp\`ere measures that
are different. Indeed, if $L$ is a generic line,
$\f_j^L:=\max(\f,u_L-j)\in PSH(\P^2,\om)$, where $\om+dd^c
u_L=[L]$, then $\f_j^L$ is continuous outside the finite set $L
\cap {\mathcal C}$, and $\varphi_j^L\searrow\varphi$. Hence
$\f_j^L \in DMA(\P^2,\om)$ and
$$
(\om+dd^c \f_j^L)^2=\om_\f\wedge [L]=\frac{1}{d}\sum_{p \in L \cap
{\mathcal C}} \d_p
$$
is independent of $j$. Here $\delta_p$ is the Dirac mass at $p$,
and the first equality follows easily since the currents involved
have local potentials which are pluriharmonic away from their
$(-\infty)$-locus. Using sequences $\f_j^L,\f_j^{L'}$, for lines
$L\neq L'$, we conclude by Proposition \ref{P:dec} that $\f \notin
DMA(\P^2,\om)$.
\end{exa}

\par The previous construction can be generalized to exhibit
examples of functions $\f\notin DMA(X,\om)$ with zero Lelong
numbers at all but one point.

\begin{exa}\label{E:Evans}
Assume $\f \in PSH(\P^2,\om)$ is such that $\{\f=-\infty\}$ is a
closed proper subset of $\P^2$, and the positive current $\om_\f$
is supported on $\{\f=-\infty\}$. We claim that $\f \notin
DMA(\P^2,\om)$. Indeed, let $p\notin\{\f=-\infty\}$ and $q_1,q_2$
be distinct points in $\{\f=-\infty\}$. The lines $L_k=(pq_k)$
intersect $\{\f=-\infty\}$ in compact subsets of $\C \subset \P^1
\simeq L_k$, hence both sequences
$$\f_j^{L_k}:=\max(\f, u_{L_k}-j)\in DMA(\P^2,\om),\;where
\;\om+dd^cu_{L_k}=[L_k],$$ decrease towards $\f$. Now $(\om+dd^c
\f_j^{L_k})^2=\om_\f\wedge [L_k]$ are distinct measures
(independent of $j$), hence $\f\notin DMA(\P^2,\om)$.\end{exa}

\par A concrete function $\f$ as in Example \ref{E:Evans} is
obtained as follows. Fix $a \in \P^2$ and a line $L \simeq \P^1$,
$a \notin L$. Let $\pi$ denote the projection from $a$ onto $L$:
this is a meromorphic map which is holomorphic in $\P^2 \setminus
\{a\}$. Fix $\nu$ a probability measure on $L$ with no atom, whose
logarithmic potential is $-\infty$ exactly on the support of $\nu$
(one can consider for instance the Evans potential of a compact
polar Cantor set). Then $\pi^* \nu=\om_\f$, where $\f \in
PSH(\P^2,\om)$ satisfies our assumptions. Note that $\nu(\f,x)=0$,
$x\neq a$, and $\nu(\f,a)=1$. Further examples can be obtained by
considering functions $\f$ such that $\om_\f$ is a laminar
current, whose transverse measure is an Evans measure.

\subsection{Special subclasses}\label{SS:subclasses}

There are several properties of $DMA(X,\om)$ that we expect to hold.
We analyze here some of these and introduce interesting
subclasses of $DMA(X,\om)$.

\subsubsection{Intermediate Monge-Amp\`ere operators}

It is natural to expect that if a function $\f \in PSH(X,\om)$
has a well defined Monge-Amp\`ere measure $(\om+dd^c \f)^n$,
then the currents $(\om+dd^c \f)^\ell$ are also well defined
for $1 \leq \ell \leq n$.
Unfortunately we are unable to prove this
(except of course in dimension 2), hence the following:

\begin{defi}\label{D:intMA} We let $DMA_\ell(X,\om)$, where $1\leq
\ell\leq n$, be the set of functions $\varphi\in PSH(X,\omega)$
for which there is a positive closed current $MA^\ell(\f)$ of
bidegree $(\ell,\ell)$ with the following property: If
$\{\varphi_j\}$ is any sequence of bounded $\omega$-psh functions
decreasing to $\varphi$ then $\om_{\f_j}^\ell\rightarrow
MA^\ell(\f)$, in the weak sense of currents. We set
$\omega_\f^\ell=(\om+dd^c \f)^\ell:=MA^\ell(\f)$. We also set
$$
DMA_{\leq k}(X,\om)=\bigcap_{\ell=1}^k DMA_\ell(X,\om).
$$
\end{defi}

\par We let the reader check that Proposition \ref{P:dec} holds
for $\f\in DMA_\ell(X,\omega)$ and $1\leq\ell\leq n$. Clearly
$$DMA_{\leq n}(X,\om)\subseteq DMA_n(X,\om)=DMA(X,\om),$$ with
equality when $n=2$  since $\f\longmapsto\om_\f$ is well defined
for all $\omega$-psh functions. We expect the equality to hold
also when $n\geq3$.
\vskip.1cm

\par It is clear from the definition that
$\varphi\in DMA(X,\omega)$ if and only if $\lambda\varphi\in
DMA(X,\lambda\omega)$, where $\lambda>0$. One would expect
moreover that, if $\f \in DMA(X,\om)$, then $\l \f \in DMA(X,\om)$
for $0 \leq \lambda \leq 1$, as the function $\l \f$ is slightly
less singular than $\f$. It is also natural to expect that the
class $DMA(X,\om)$ is stable under taking maximum:
$$
\f \in DMA(X,\om) \text{ and } \p \in PSH(X,\om)
\stackrel{?}{\Longrightarrow } \max(\f,\p) \in DMA(X,\om).
$$
If this property holds, then applying it to $\p=\l \f \leq 0$, $0 \leq \l \leq 1$, shows that
$\l \f=\max(\f,\l \f) \in DMA(X,\om)$ as soon as $\f \in DMA(X,\om)$.
An alternative and desirable property is
$$
DMA(X,\om) \stackrel{?}{=} DMA(X,\om') \cap PSH(X,\om),
$$
when $\om \leq \om'$.
All these properties are related to convexity properties of
$DMA(X,\om)$.

\subsubsection{The non pluripolar part}

\par  It was observed in \cite{GZ2} that if $\f$ is $\om$-psh and
$\f_j:=\max(\f,-j)$, then
$$j\longmapsto\mu_j(\f):=\1_{\{\f >-j \}} (\om+dd^c \f_j)^n$$
is an increasing sequence of positive Radon measures, of total
mass uniformly bounded above by $\int_X \om^n$. Thus
$\{\mu_j(\f)\}$ converges to a positive Radon measure $\mu(\f)$ on
$X$. It is generally expected (see \cite{BT87} for similar
considerations in the local context) that $\mu(\f)$ should
correspond to the non pluripolar part of $(\om+dd^c\f)^n$,
whenever the latter makes sense. We can justify this expectation
in two special cases:

\begin{pro}\label{P:P16}
Fix $\f \in DMA(X,\om)$. Then
$$
\mu(\f)\leq{\bf 1}_{ \{ \f>-\infty \} } (\om+dd^c \f)^n,
$$
with equality if $\exp \f$ is continuous, or if $\om_\f^n$ is
concentrated on $\{\f=-\infty\}$.
\end{pro}

\begin{proof}
We can assume wlog that $\f \leq 0$.
Set $u_s:=\max(\f/s+1,0)$. Note that $u_s$ are bounded $\om$-psh functions
which increase towards ${\bf 1}_{\{\f >-\infty\}}$.
Moreover $\{u_s>0\}=\{\f>-s\}$ and $u_s=0$ elsewhere. We infer, when $j>s$, that
$$
u_s (\om+dd^c \f_j)^n=u_s {\bf 1}_{ \{ \f>-j \}} (\om+dd^c \f_j)^n,
$$
where $\f_j=\max(\f,-j)$ are the canonical approximants.

When $e^{\f}$ is continuous, then so is $u_s$, hence passing to
the limit yields $u_s (\om+dd^c \f)^n=u_s \mu(\f)$. Letting $s
\rightarrow +\infty$, we infer $\mu(\f)={\bf 1}_{ \{ \f>-\infty \}
} \mu(\f)={\bf 1}_{\{\f>-\infty\}}(\om+dd^c \f)^n$.

In the general case, we only get $u_s \mu(\f) \leq u_s (\om+dd^c
\f)^n$, since $u_s$ is upper-semi-continuous. This yields
$\mu(\f)={\bf 1}_{ \{ \f>-\infty \} } \mu(\f) \leq {\bf 1}_{ \{
\f>-\infty \} } (\om+dd^c \f)^n$, whence equality if $(\om+dd^c
\f)^n$ is concentrated on $\{\f=-\infty\}$.
\end{proof}

To overcome this difficulty in the general case, we introduce
interesting subclasses of $DMA(X,\om)$:

\begin{defi}\label{D:DMAhat} Fix $1 \leq \ell \leq n$.
We let $\widehat{DMA}_\ell(X,\om)$ (resp. $\widehat{DMA}_{\leq \ell}(X,\om)$)
denote the set of functions $\f \in DMA_\ell(X,\om)$
(resp. $DMA_{\leq \ell}(X,\om)$) such that
for any sequence $\f_j \in PSH(X,\om) \cap L^{\infty}(X)$ decreasing to $\f$,
$$
\int_X u (\om+dd^c \f_j)^\ell \wedge \om^{n-\ell} \longrightarrow
\int_X u (\om+dd^c \f)^\ell \wedge \om^{n-\ell},
$$
for all $u \in PSH(X,\om) \cap L^{\infty}(X)$.
\end{defi}

Note that this convergence property is stronger than the usual
convergence in the weak sense of Radon measures: any smooth test
function is $C\om$-psh for some constant $C>0$. The following
corollary can be proved exactly like Proposition \ref{P:P16}.

\begin{cor}\label{C:C18} If $\f \in \widehat{DMA}(X,\om)$ then
$$
\mu(\f)={\bf 1}_{\{\f>-\infty\}} (\om+dd^c \f)^n.
$$
Moreover, if $\f_j= \max(\f , - j)$ and $B \subset \{ \f >  -
\infty\}$ is any Borel set then
$$
\int_B \omega_{\f}^n = \lim_{j \to + \infty} \int_{B \cap \{\f > -
j\}} \omega_{\f_j}^n.
$$
\end{cor}

As a consequence, we obtain the following generalized comparison principle:

\begin{thm}\label{T:comparison} Let  $\f, \p\in PSH (X,\om)$ and set $\f \vee \psi := \max(\f,\psi)$. Then
$$
{\bf 1}_{\{ \f > \psi\}} \mu ( \f) ={\bf 1}_{\{\f > \psi\}}
\mu (\f \vee \psi).
$$
Moreover if $\f, \p\in \widehat{DMA} (X,\om)$ then
$$
\int _{\{  \f < \psi \} } \omega_{\psi}^n \leq \int _{\{  \f <
\psi \} \cup \{\f = -\infty\} }\omega_{\f}^n.
$$
\end{thm}

\begin{proof}
 Set $\f_j = \max (\f, -j ) $
and $\psi_j =\max (\psi, -j )$. Recall from \cite{BT87} that the
desired equality is
 known for bounded psh functions,
$$
{\bf 1}_{\{\f_j> \psi_{j+1}\}} (\omega + dd^c \f_j )^n ={\bf 1}_{\{\f _j> \psi_{j+1}\}}
 (\omega + dd^c \max ( \f_j , \psi_{j+1}) )^n .
$$
Observe that
 $\{\f > \psi\} \subset \{\f _j > \psi_{j+1}\}$, hence
\begin{multline*}
{\bf 1}_{\{\f > \psi \}}\cdot {\bf 1}_ {\{\f  >-j \} } (\omega + dd^c \f_j )^n =
{\bf 1}_{\{\f > \psi \}}\cdot {\bf 1}_ {\{\f  >-j \} }  (\omega + dd^c \max (\f , \psi , -j) )^n
\\
 =
{\bf 1}_{\{\f > \psi \}} \cdot {\bf 1}_{ \{\f \vee \psi > -j \} }
(\omega + dd^c \max ( \f \vee \psi , -j) )^n  .
\end{multline*}
Note that the sequence of measures $ {\bf 1}_ {\{\f  >-j \} } (\omega + dd^c \f_j )^n $
converges in the strong sense of Borel measures towards $\mu(\f) $ and
 the sequence  ${\bf 1}_{ \{ \f \vee \psi > -j \}}
(\omega + dd^c ( \f \vee \psi , -j) )^n$ converges in the strong
sense of Borel measures towards $\mu (\f \vee \psi) $. Hence,
since $\{\f > \p\} \subset  \{\f > - \infty\} \subset \{ \f \vee
\p > - \infty\},$ it follows from Corollary \ref{C:C18} that
$$
{\bf 1}_{\{\f > \psi \}} \mu (\f)  ={\bf 1}_{\{\f > \psi\}}
\mu (\f \vee \psi) .
$$
Now let  $\f, \p\in  \widehat{DMA} (X,\om)$ and assume first that
$\p$ is bounded. Then it follows from Corollary \ref{C:C18} that
$$
\int _{\{\f < \psi\}} (\omega + dd^c \psi )^n
 =  \int _{\{\f < \psi\}} (\omega + dd^c \max(\psi,\f) )^n.
$$
Since $ \int_X (\omega + dd^c \max(\psi,\f) )^n =  \int_X (\omega + dd^c \f )^n$ it follows that
$$
\int _{\{\f < \psi\}} (\omega + dd^c \psi )^n  \leq   \int _X (\om+dd^c \f )^n -
\int _{\{\f > \psi \}} (\om+dd^c \f )^n  \\
=  \int _{\{\f \leq \psi\}} (\om+dd^c \f)^n.
$$
If $\p \in \widehat{DMA} (X,\om)$ is not bounded, we apply the
previous inequality to $\f$ and $\p_j := \max(\p , - j)$ for $j
\in \N$. Then we get for any $j \in \N$
$$ \int _{\{\f < \psi_j \}} (\omega + dd^c \psi_j )^n  \leq  \int _{\{\f \leq \psi_j\}} (\om+dd^c \f)^n.$$
Since $\{\f < \p \} \cap \{\p > - j\} \subset \{\f < \p_j \}$ for
any $j$ and $\{\f \leq \psi_j\}$ is a decreasing sequence of Borel
sets converging to the Borel set $\{\f \leq \p \}$,  by taking the
limit in the previous inequalities and applying Corollary
\ref{C:C18}, we obtain
$$ \int _{\{\f < \psi\}} (\omega + dd^c \psi )^n  \leq \int _{\{\f \leq \psi\}} (\om+dd^c \f)^n.$$ Now take  $0<
\varepsilon <1$ and apply the previous result with $\f + \e$ and
$\psi$. Then let $\e \to 0$ to obtain  the required inequality.
\end{proof}

\begin{rqe} Let $\f\in PSH(X,\om)$ and
$\f_j:=\max(\f,-j)$. It follows as in \cite{GZ2} that
$\mu_j^\ell(\f)=\1_{\{\f >-j \}} (\om+dd^c \f_j)^\ell$ is an
increasing sequence of positive currents of mass at most $\int_X
\om^n$, so it converges in the strong sense to a positive current
$\mu^\ell(\f)$ on $X$. We leave it to the reader to check that
Proposition \ref{P:P16}, Corollary \ref{C:C18} and Theorem
\ref{T:comparison} have analogues in the classes
$DMA_\ell(X,\om)$. For instance, if $\f\in DMA_\ell(X,\om)$ and
$e^\f$ is continuous, or if $\f\in \widehat{DMA}_\ell(X,\om)$,
then $\mu^\ell(\f)={\bf 1}_{\{\f>-\infty\}}\om_\f^\ell$. Moreover,
if $\f, \p\in \widehat{DMA}_\ell(X,\om)$ then
$$\int _{\{  \f < \psi \}
} \omega_{\psi}^\ell\wedge\om^{n-\ell} \leq \int _{\{ \f < \psi \}
\cup \{\f = -\infty\} }\omega_{\f}^\ell\wedge\om^{n-\ell}.$$
\end{rqe}

\begin{rqe}
Observe that in dimension $n=1$,
$$
PSH(X,\om)=DMA(X,\om)=\widehat{DMA}(X,\om).
$$
The latter equality follows easily by integrating by parts and
using the monotone convergence theorem. We expect that the
equality $\widehat{DMA}(X,\om)=DMA(X,\om)$ continues to hold in
higher dimension.
\end{rqe}

\subsubsection{The class ${\mathcal E}(X,\om)$}

We consider here
$$
{\mathcal E}(X,\om)=\left\{\f\in
PSH(X,\om)\,/\,\mu(\f)(X)=\int_X\om^n\right\}.
$$
Equivalently, $\f \in {\mathcal E}(X,\om)$ if and only if
$(\om+dd^c \f_j)^n(\{\f\leq-j\})\rightarrow 0$. This class of
functions was studied in \cite{GZ2}, where it was shown that
\par $(i)$ ${\mathcal E}(X,\om) \subset DMA(X,\om)$;
\par $(ii)$ ${\mathcal E}(X,\om)$ is convex and stable under maximum;
\par $(iii)$ ${\mathcal E}(X,\om)$ is the largest subclass of
$DMA(X,\om)$ on which the comparison principle holds.

The rough idea is that a $\om$-psh function $\f$ should
belong to ${\mathcal E}(X,\om)$
iff it belongs to $DMA(X,\om)$ and its Monge-Amp\`ere measure does not
charge pluripolar sets.
This is partly justified by the following result:

\begin{pro}\label{P:P112}
A function $\f$ belongs to ${\mathcal E}(X,\om)$ if and only if
it belongs to $\widehat{DMA}(X,\om)$ and
$(\om+dd^c \f)^n$ does not charge pluripolar sets.
\end{pro}

\begin{proof}
The inclusion ${\mathcal E}(X,\om) \subset \widehat{DMA}(X,\om)$
will follow from Theorem \ref{T:econvergence} below. Assume
conversely that $\f \in \widehat{DMA}(X,\om)$ is such that
$(\om+dd^c \f)^n$ does not charge pluripolar sets. Then
$$
\mu(\f)={\bf 1}_{\{\f>-\infty\}} (\om+dd^c \f)^n=(\om+dd^c \f)^n
$$
has full mass, hence $\f \in {\mathcal E}(X,\om)$.
\end{proof}

\begin{rqe}
The previous characterization of ${\mathcal E}(X,\om)$ is related
to the question of uniqueness of solutions to the equation
$(\om+dd^c \cdot)^n=\mu$. Indeed assume $\f \in DMA(X,\om)$ is
such that $\mu:=(\om+dd^c \f)^n$ does not charge pluripolar sets.
It was shown in \cite{GZ2} that there exists $\p \in {\mathcal
E}(X,\om)$ so that $\mu=(\om+dd^c \p)^n$. It is expected that the
solution $\p$ is unique up to an additive constant. If such is the
case, then $\f \equiv \p+constant$ belongs to ${\mathcal
E}(X,\om)$.
\end{rqe}

\section{Finite energy classes}\label{S:E}

In this section we establish further properties of the class
${\mathcal E}(X,\om)$ and we consider energy classes with respect
to a fixed current $T$.

\subsection{Weighted energies}\label{SS:WE}

Let $T$ be a positive closed current of bidimension $(m,m)$ on
$X$. In the sequel $T$ will be of the form $T=(\om+dd^c u_{m+1})
\wedge \cdots \wedge (\om+dd^c u_n)$, where $u_j \in PSH(X,\om)
\cap L^{\infty}(X)$. Let $\chi:\R \rightarrow \R$ be a convex
increasing function such that $\chi(-\infty)=-\infty$. Following
\cite{GZ2} we consider
$$
{\mathcal E}_{\chi}(T,\om):=\left\{ \f \in PSH(X,\om) \, / \,
\sup_j \int_X (-\chi) \circ \f_j
\, \om_{\f_j}^m \wedge T <+\infty \right\},
$$
where $\f_j:=\max(\f,-j)$ denote the canonical approximants of
$\f$. We let the reader check that \cite{GZ2} can be adapted line
by line, showing that the Monge-Amp\`ere measure $(\om+dd^c \f)^m
\wedge T$ is well defined for $\f \in {\mathcal E}_{\chi}(T,\om)$
and that $\chi \circ \f \in L^1((\om+dd^c \f)^m \wedge T)$.

When $m=n$, i.e. $T=[X]$ is the current of integration along $X$, then
the classes
${\mathcal E}_{\chi}(T,\om)={\mathcal E}_{\chi}(X,\om)$ yield
the following alternative description
of ${\mathcal E}(X,\om)$
(see \cite[Proposition 2.2]{GZ2}):
$$
{\mathcal E}(X,\om)=\bigcup_{\chi \in {\mathcal W}^-} {\mathcal
E}_{\chi}(X,\om),\; {\mathcal E}_{\chi}(X,\om):=\left\{ \f
\in{\mathcal E}(X,\om): \, \chi \circ \f \in L^1(\om_\f^n)
\right\},$$ where ${\mathcal W}^-=\{\chi:\R \rightarrow
\R\,/\,\chi \;{\rm
convex,\;increasing},\;\chi(-\infty)=-\infty\}$. We set
$$
{\mathcal E}(T,\om):=\bigcup_{\chi \in {\mathcal W}^-}
{\mathcal E}_{\chi}(T,\om).
$$

\begin{thm}\label{T:econvergence}
Fix $\chi \in {\mathcal W}^-$.
Let $\f_j$ be a sequence of $\om$-psh functions decreasing
towards $\f \in {\mathcal E}_{\chi}(T,\om)$. Then $\f_j \in
{\mathcal E}_{\chi}(X,\om)$ and
$$
\int_X (-\chi) \circ \f_j \, (\om+dd^c \f_j)^m \wedge T
\longrightarrow \int_X (-\chi) \circ \f \, (\om+dd^c \f)^m \wedge T.
$$
Moreover
for any $u \in PSH(X,\om) \cap L^{\infty}(X)$,
$$
\int_X u (\om+dd^c \f_j)^m \wedge T \longrightarrow
\int_X u (\om+dd^c \f)^m \wedge T.
$$
\end{thm}

\begin{proof}
When $\f_j$ is the canonical sequence of approximants, the theorem
follows by a similar argument as in \cite[Theorem 2.6]{GZ2}.
Moreover, it also follows  from \cite[Theorem 2.6]{GZ2} that the
first convergence in the statement holds for an arbitrary
decreasing sequence, if there exists a weight function $\tilde
\chi$ such that $\chi = o (\tilde \chi)$ and $\f \in \mathcal
E_{\tilde \chi} (T,\om)$. It turns out that such a weight always
exists. Indeed, since $\chi (\f) \in L^1 (\omega_{\f}^m \wedge
T)$, it follows from standard measure theory arguments that there
exists a convex increasing function $h : \R^+ \to \R^+$ such that
$\lim_{t \to + \infty} h (t) \slash t = + \infty$ and $ h (- \chi
(\f)) \in L^1 (\omega_{\f}^m \wedge T)$. Note that we can assume
that $h$ has polynomial growth, which implies that $\f \in
\mathcal E_{\tilde \chi} (T,\om)$, where $\tilde \chi := - h (-
\chi) \in {\mathcal W}^- \cup {\mathcal W}_M^+$ (see \cite{GZ2}
for the definition of this latter class). \vskip.1cm

We now prove the second assertion. Set $\f_j^k := \max(\f_j , -
k)$ and $\f^k := \max(\f , - k)$. Then
\begin{eqnarray*}
\int_X u \, \omega_{\f_j}^m \wedge T -  u \, \omega_{\f}^m \wedge T
 &=& \int_X u \, \omega_{\f_j}^m \wedge T -  \int_X  u \, \omega_{\f_j^k}^m \wedge T \\
 &+& \int_X u \, \omega_{\f_j^k}^m \wedge T- \int_X  u \, \omega_{\f^k}^m \wedge T \\
& + & \int_X  u \, \omega_{\f^k}^m \wedge T - \int_X  u \, \omega_{\f}^m \wedge T.
\end{eqnarray*}
 It suffices to prove that
$\int_X u \, \omega_{\f_j}^m \wedge T -  \int_X  u \, \omega_{\f_j^k}^m \wedge T $
converges to $0$ as $k \to + \infty$ uniformly in $j$.
This is a consequence of the following estimate,
$$
\left|  \int_X u \, \omega_{\f_j}^m \wedge T -  \int_X  u \, \omega_{\f_j^k}^m \wedge T \right|
\leq \frac{ 2 \Vert u \Vert_{L^{\infty} (X)}}{\vert \chi (- k)\vert} \int_{X}( - \chi (\f_j))
\omega_{\f_j}^m \wedge T.
$$
The latter integral is uniformly bounded since $\f \in {\mathcal E}_{\chi}(T,\om)$.
\end{proof}

Let ${\mathcal E}^1(\om^{n-p},\om)$ denote the class ${\mathcal
E}_{\chi}(T,\om)$ for $T=\om^{n-p}$ and $\chi(t)=t$.

\begin{cor}\label{C:econvergence} Let  $1 \leq p \leq n - 1$ and
let $\{\f_j\}$ be a
decreasing sequence of $\om$-psh functions converging to $\f \in
\mathcal E^1 (\omega^{n - p},\omega)$.  Then for any $u \in
PSH(X,\omega) \cap L^{\infty} (X)$ we have
$$
\lim_{j \to + \infty} \int_X u \,\omega_{\f_j}^{p + 1} \wedge
\om^{n - p - 1} =  \int_X u \,\omega_{\f}^{p + 1} \wedge \om^{n -
p - 1}.
$$
In particular, $\mathcal E^1 (\omega^{n - p},\omega) \subset
\widehat {DMA}_{p + 1}(X,\om)$.
\end{cor}

\begin{proof} For simplicity, we consider the case $p = n - 1$.
The general case follows along the same lines. We want to prove
that if $\f_j\searrow\f \in \mathcal E^1 (\omega,\omega)$ then for
any $u \in PSH(X,\omega) \cap L^{\infty} (X)$ we have

$$ \lim_{j \to + \infty} \int_X u\, \omega_{\f_j}^{n} =
\int_X u \, \omega_{\f}^n.$$

Observe that since $\omega_{\f_j}^n \to \omega_{\f}^n$ in the weak
sense of Radon measures on X,  the above equality holds when $u$
is continuous on $X$.

We claim that  for any  $\psi \in \mathcal E^1 (\omega,\omega)$,
we have
\begin{equation}\label{e:intpartsE1}
\int_X (-  u) \omega_{\psi}^n = \int_X(- u) \omega \wedge
\omega_{\psi}^{n - 1}
 + \int_X (- \psi) \omega_u  \wedge \omega_{\psi}^{n - 1} + \int_X \psi
 \omega \wedge \omega_{\psi}^{n - 1}.
\end{equation}
Observe that this identity is just integration by parts which
clearly holds when $u$ is a smooth test function on $X$. The
identity also holds when $u, \psi \in PSH (X,\om) \cap L^{\infty}
(X)$ (see \cite{GZ2}).

Fix $u$ a bounded $\om$-psh function and set $\psi_j := \max
\{\psi , - j\}$ for $j \in \N$. Applying (\ref{e:intpartsE1}) to
$u$ and $\psi_j$, it follows immediately that $\mathcal E^1
(\omega_u,\omega) = \mathcal E^1 (\omega,\omega)$. Hence the
corollary follows at once from Theorem \ref{T:econvergence}, as
soon as (\ref{e:intpartsE1}) is established for $\psi \in \mathcal
E^1 (\omega,\omega)$.

To prove (\ref{e:intpartsE1}) we can assume that $u \leq 0$. Note
that $\om_{\psi_j}^n \to\om_{\psi}^n$ weakly. Using the upper
semicontinuity of $u$ and applying (\ref{e:intpartsE1}) to $u$ and
$\psi_j$, it follows from Theorem 2.1  that
\begin{eqnarray}\label{e:intpartsineq1}
&&\int_X (-  u) \omega_{\psi}^n\leq\liminf_{j \to + \infty}
\int_X (-  u) \omega_{\psi_j}^n= \\
&& \int_X  (- u)
 \omega \wedge \omega_{\psi}^{n - 1} + \int_X (- \psi) \omega_u  \wedge
 \omega_{\psi}^{n - 1} + \int_X \psi  \omega \wedge \omega_{\psi}^{n -
 1}.\nonumber
 \end{eqnarray}

Next let $u_j\searrow u$ be a sequence of smooth $\om$-psh
functions (see \cite{BK07}). Note that $ u_j \,\omega_{\psi}^{n -
1} \to  u \,\omega_{\psi}^{n - 1}$ in the weak sense of currents,
hence $ \omega_{u_j} \wedge \omega_{\psi}^{n - 1} \to
\omega_{u}\wedge \omega_{\psi}^{n - 1}$ weakly in the sense of
measure. Applying (\ref{e:intpartsE1}) to $u_j$ and $\psi$, it
follows by monotone convergence and the upper semicontinuity of
$\psi$ that
\begin{eqnarray}\label{e:intpartsineq2}
&&\int_X (- u) \omega_{\psi}^n= \\
&&\int_X  (- u) \omega \wedge \omega_{\psi}^{n - 1}+\lim_{j \to +
\infty} \int_X (-\psi) \omega_{u_j} \wedge \omega_{\psi}^{n - 1} +
\int_X\psi\omega \wedge \omega_{\psi}^{n - 1}\geq\nonumber\\
&&\int_X  (- u) \omega \wedge \omega_{\psi}^{n - 1}+\int_X (-\psi)
\omega_u \wedge \omega_{\psi}^{n - 1} + \int_X \psi\omega \wedge
\omega_{\psi}^{n - 1}.\nonumber
\end{eqnarray}

The identity (\ref{e:intpartsE1}) now follows from the
inequalities (\ref{e:intpartsineq1}) and (\ref{e:intpartsineq2}).
\end{proof}

\begin{cor}\label{C:EDMAhat}
$
{\mathcal E}(X,\om) \subset \widehat{DMA}_{\leq n}(X,\om).
$
\end{cor}

\begin{proof} Fix $\f \in {\mathcal E}(X,\om)$.
Then there exists $\chi \in {\mathcal W}^-$ such that $\f \in {\mathcal E}_{\chi}(X,\om)$.
Recall now that for any $1 \leq p \leq n-1$,
$$
{\mathcal E}_{\chi}(X,\om) \subset {\mathcal E}_{\chi}(\om^p,\om),
$$
as follows from simple integration by parts (see \cite{GZ2}). We
can thus apply Theorem \ref{T:econvergence} several times to
conclude.
\end{proof}

\subsection{Twisted Monge-Amp\`ere operators}\label{SS:twisted}
We show here that certain Monge-Amp\`ere operators with weights
are always well defined.

\begin{thm}\label{T:tma1}
Let $\eta:\R\rightarrow \R^+$ be a continuous function with
$\eta(-\infty)=0$ and $1 \leq \ell \leq n$. Let $\f \in
PSH(X,\om)$ and $\{\f_j\}$ be any sequence of bounded $\om$-psh
functions decreasing to $\f$. Then the twisted Monge-Amp\`ere currents
$$
M_{\eta}^{\ell}(\f_j):=\eta \circ \f_j (\om+dd^c \f_j)^{\ell}
$$
converge weakly towards a positive current $M_{\eta}^{\ell}(\f)$,
which is independent of the sequence $\{\f_j\}$. The twisted
Monge-Amp\`ere operator $M_{\eta}^{\ell}$ is well defined on
$PSH(X,\om)$ and is continuous along any decreasing sequences of
$\om$-psh functions. If $\f \in \widehat{DMA}_{\ell}(X,\om)$ then
$$
M_{\eta}^{\ell}(\f)= \eta\circ\f\, (\om+dd^c \f)^\ell.
$$
\end{thm}

\begin{proof}
It is a standard fact that the operator
$M_{\eta}^{\ell}(\f)= \eta\circ\f\, (\om+dd^c \f)^\ell$
is well defined and continuous under decreasing sequences in
the subclass of bounded $\omega$-psh functions \cite{BT82}.

\par As it was observed in Remark 1.10,  the sequence
$$k\longmapsto{\bf 1}_{\{\varphi>-k\}}(\omega +
dd^c \max(\varphi,-k))^\ell$$ is an increasing sequence of
positive currents of total mass bounded by $\int_X \om^n$. Hence
this sequence converges in a strong sense to a positive current
$\mu^{\ell}(\varphi)$ which puts no mass on pluripolar sets and
satisfies (see \cite[Theorem 1.3]{GZ2})
\begin{equation}\label{e:nu}
{\bf 1}_{\{\varphi>-k\}}(\omega + dd^c \max (\varphi,-k))^\ell  =
{\bf 1}_{\{\varphi > - k\}} \mu^{\ell} (\varphi).
\end{equation}
 Since $\eta \circ \varphi$  is a bounded positive Borel function on $X$,
we can define a positive current on $X$
$$
M_{\eta}^{\ell}(\f) :=
\eta(\varphi)\,\mu^{\ell} (\varphi).
$$
Note that when $\varphi$ is bounded then
$\mu^{\ell}(\varphi) =\omega_\varphi^{\ell} $ hence $M_{\eta}^{\ell}(\f) =
\eta(\varphi)\ \omega_\varphi^{\ell}$.

To prove the continuity of this operator, let $\f_j$ be any
sequence of $\om$-psh functions decreasing to $\f$, and set
$\varphi_{j}^k := \max(\varphi_{j} , - k)$, $\varphi^k := \max
(\varphi , - k)$ for $j,k \in \N$.

Since the forms of type $(n-\ell,n-\ell)$ on $X$ have a basis
consisting of forms $h \Psi$, where $h$ are smooth functions
and $\Psi$ are positive closed forms, it suffices to prove that
$\eta(\varphi_j)\,\mu^{\ell} (\varphi_j) \wedge \Psi \to \eta(\varphi)\,
\mu^{\ell} (\varphi) \wedge \Psi$
weakly as measures on $X$.
For simplicity, we may assume that $\Psi=\omega^{n-\ell} $.
Then
\begin{eqnarray*}
 && \int h\,\eta(\varphi_j)\, \mu^{\ell} (\varphi_j) \wedge \om^{n - \ell}  -
\int h\,\eta(\varphi)\,\mu^{\ell} (\f) \wedge \om^{n - \ell}  \\
&& =  \int
h\,\eta(\varphi_j)\,\mu^{\ell} (\f_j) \wedge \om^{n - \ell} -
\int h\,\eta(\varphi_j^k)\, \mu^{\ell} ({\varphi_j^k}) \wedge \om^{n - \ell}   \\
&&  + \int h\,\eta(\varphi_j^k)\, \mu^{\ell}({\varphi_j^k}) \wedge
\om^{n - \ell} -
\int h\,\eta(\varphi^k)\, \mu^{\ell}({\varphi^k}) \wedge \om^{n - \ell}   \\
&& + \int h\,\eta(\varphi^k)\, \mu^{\ell}({\varphi^k})\wedge \om^{n
- \ell} - \int h\,\eta(\varphi)\, \mu^{\ell}({\varphi}) \wedge
\om^{n - \ell}.
\end{eqnarray*}
We claim that the first term in the sum tends to $0$ uniformly in
$j$ as $k \to + \infty.$ Indeed, we have by (\ref{e:nu})
\begin{eqnarray*}
&& \left|\int h\,\eta(\varphi_j)\, \mu^{\ell}({\varphi_j}) \wedge
\om^{n - \ell} - \int h\,\eta(\varphi_j^k)\, \mu^{\ell}
({\varphi_j^k}) \wedge \om^{n - \ell}\right|
\\
&& \leq  \int_{\{\varphi_j \leq - k\}}|h|\,\eta (\varphi_j)\,
\mu^{\ell} ({\varphi_j}) \wedge \om^{n - \ell} + \int_{\{\varphi_j
\leq -
k\}}|h|\,\eta(\varphi_j^k)\, \mu^{\ell} ({\varphi_j^k}) \wedge \om^{n - \ell} \\
&& \leq  2\tilde\eta(-k)\|h\|_\infty\int_X \omega^n,
\end{eqnarray*}
where $\tilde{\eta}(-k):=\sup\{\eta(s):\,s\leq-k\}\to0$ as $k\to+
\infty$. In the same way we see that the last term tends to $0$ as
$k \to + \infty$. Now for fixed  $k$, it follows from the uniformly bounded case
that the second term tends to $0$ as $j \to + \infty$. The
desired continuity result follows.
\end{proof}

\begin{thm}\label{T:tma2} Let $\eta:\R\rightarrow\R^+$ be an
increasing function of class $C^1$ and $0\leq\ell\leq n-1$. Let
$\f \in PSH(X,\om)$ and $\{\f_j\}$ be any sequence of bounded
$\om$-psh functions decreasing to $\f$. Then the  currents
$$
S^{\ell}_\eta(\f_j):=\eta'\circ\f_j\,d \f \wedge d^c \f_j \wedge
(\omega+dd^c\f_j)^{n-\ell-1}
$$
converge weakly towards a positive current $S_{\eta}^{\ell}(\f)$
which is independent of the sequence $\{\f_j\}$. The operator
$S_{\eta}^{\ell}$ is well defined on $PSH(X,\om)$ and is
continuous along any decreasing sequences of $\om$-psh functions.
\end{thm}

\begin{proof} By subtracting a constant, we may assume that
$\eta(-\infty)=0$. Fix $\ell$ and $\f\in PSH(X,\omega)$. We can
assume that $\f<0$ and $\eta(0)=1$.

\par Observe that  if $\theta :\R \longrightarrow \R$ is any $C^1$ function with
$0 \leq \theta \leq 1$ then for any
$u\in PSH(X,\omega)\cap L^\infty(X)$ and
 any positive closed current $R$ on $X$ of bidimension $(1,1)$, we
have
$$d(\theta (u)d^c u \wedge R)=\theta' (u)du\wedge d^cu\wedge R+
\theta (u) \omega_u\wedge R - \theta (u) \omega \wedge R.$$
 Using Stokes'
formula and the fact that $0\leq \theta  (u)\leq 1$, we get the following uniform bound
\begin{equation}\label{e:bdmass}
\int_X \theta' (u)\,du\wedge d^cu\wedge R\leq\int_X \omega\wedge
R.\end{equation}

\par Assume that $\f \leq 0$ and  set $\f^k := \max(\f,-k)$ for $k \geq 0$. We want
to prove that  the sequence of
positive currents
$${\bf 1}_{\{ \f > -
k\}} \eta'(\f^k)\, d \f^k \wedge d^c \f^k  \wedge \om_{\f^k}^{n - \ell - 1}$$
converges to a positive current of bidimension $(\ell,\ell)$ which will be denoted
by $S_{\eta}^{\ell} (\f)$. It follows from
the quasi-continuity of bounded psh functions \cite{BT87} that
for $j \geq k \geq 0$,

$$
{\bf 1}_{\{ \f > - k\}}  \eta' (\f^k)   d \f^k \wedge d^c \f^k  \wedge \om_{\f^k}^{n - \ell - 1} =
{\bf 1}_{\{ \f > - k\}}  \eta' (\f^j)  d \f^j \wedge d^c \f^j  \wedge \om_{\f^j}^{n - \ell - 1}.
$$

This implies that
$$ {\bf 1}_{\{ \f > -
k\}}  \eta' (\f^k) d \f^k \wedge d^c \f^k  \wedge \om_{\f^k}^{n -
\ell - 1}$$ is an increasing sequence of positive currents of
bidimension $(\ell , \ell)$, with uniformly bounded mass $\leq
\int_X \om^n$ by (\ref{e:bdmass}). Therefore it converges in a
strong sense to a positive current $S_{\eta}^{\ell} (\f)$ on $X$
which satisfies the following equation
\begin{equation}\label{e:jk}
{\bf 1}_{\{ \f > - k\}}  S_{\eta}^{\ell} (\f) ={\bf 1}_{\{ \f > - k\}}
\eta' (\f^k) d \f^k \wedge d^c \f^k  \wedge \om_{\f^k}^{n - \ell - 1} .
\end{equation}
Observe that if $\f$ is bounded then
$S_{\eta}^{\ell} (\f) = \eta' (\f) d \f \wedge d^c \f  \wedge \om_{\f}^{n - \ell - 1}$.

Now we want to prove that for any decreasing sequence $\{\f_j\}$
which converges to $\f$, the currents $S_{\eta}^{\ell} (\f_j)$
converge to the current $S_{\eta}^{\ell} (\f)$ in the sense of
currents on $X$. As before it is enough to prove that the sequence
of positive measures $ S_{\eta}^{\ell} (\f_j) \wedge \om^{\ell}$
converges weakly to the positive measure $S_{\eta}^{\ell} (\f)
\wedge \om^{\ell}$ on $X$. Let $h$ be a continuous function on
$X$. Proceeding as in the proof of the previous theorem, it
suffices to show that
$$
\left|\int h\, S_{\eta}^{\ell} (\f_j) \wedge \om^{\ell} - \int h\, \eta' (\f_j^k) d \f_j^k
\wedge d^c \f_j^k  \wedge \om_{\f_j^k}^{n - \ell - 1}\wedge\om^{\ell}\right| \to 0
$$
as $k\to\infty$, uniformly in $j$, where $\f_j^k=\max(\f_j,-k)$.  Indeed, if $\theta =\sqrt{\eta}$
then  $\eta'=2 \sqrt{\eta}\,\theta'$, so it
follows at once from the definition of $S^\ell_\eta$ that
$S^\ell_\eta(\f_j)= 2 \sqrt{\eta (\f_j)}S^\ell_\theta(\f_j)$. Since
$\eta' (\f_j^k) \leq 2 \sqrt{\eta (- k)} \theta' (\f_j^k)$ on the set $\{\f_j\leq-k\}$,
we have by (\ref{e:jk}) that
\begin{eqnarray*}
&& \left|\int h\, S_{\eta}^{\ell} (\f_j) \wedge \om^{\ell} -
 \int h\, \eta' (\f_j^k) d \f_j^k \wedge d^c \f_j^k  \wedge \om_{\f_j^k}^{n - \ell - 1}\wedge \om^{\ell}
\right|
\\
&& \leq 2 \sqrt{\eta (- k)} \int_{\{\varphi_j \leq - k\}} |h|\, S_{\theta}^{\ell} (\f_j) \wedge
\om^{\ell} \\
&&  + 2 \sqrt{\eta (- k)}  \int_{\{\varphi_j \leq - k\}} |h|\,\theta' (\f_j^k) d \f_j^k
\wedge d^c \f_j^k  \wedge \om_{\f_j^k}^{n - \ell - 1}\wedge \om^{\ell} \\
&& \leq  4 \sqrt{\eta (- k)} \|h\|_{\infty} \int_X \omega^n,
\end{eqnarray*}
which tends to $0$ as $k \to + \infty$ uniformly in $j$.
\end{proof}

\subsection{Attenuation of singularities}\label{SS:attenuation}

\par The goal of this section is to show that a very small
attenuation of singularities transforms a function $\f \in
PSH(X,\om)$ into a function $\chi \circ \f \in {\mathcal
E}(X,\om)$. In particular most functions in the class ${\mathcal
E}(X,\om)$ do not belong to the local domain of definition
$DMA_{loc}(X,\om)$, which we consider in Section \ref{S:local}.

\par Let $\chi:\R^- \rightarrow \R^-$ be a convex increasing
function of class $C^2$ and such that $\chi(-\infty)=-\infty$,
$\chi'(-\infty)=0$ and $\chi'(-1)\leq 1$. Define for $x\leq-1$
$$\eta_\ell(x)=(\chi'(x))^\ell(1-\chi'(x))^{n-\ell}\;,\;
\e_\ell(x)=\int_0^{\chi'(x)}t^{n-\ell-1}(1-t)^\ell dt.$$

\begin{cor}\label{C:atten}
If $\f \in PSH(X,\om)$, $\f\leq-1$, then $\chi \circ \f \in
{\mathcal E}(X,\om)$ and
\begin{equation}\label{e:atten}\om_{\chi \circ \f}^n=
\sum_{\ell=0}^n{n\choose\ell}M^\ell_{\eta_\ell}(\f)\wedge
\om^{n-\ell}+n\sum_{\ell=0}^{n-1}{n-1\choose\ell}
S^\ell_{\e_\ell}(\f)\wedge \om^\ell,\end{equation} where
$MA^\ell_{\eta_\ell}$ and $S^\ell_{\e_\ell}$ are the operators
defined in Theorems \ref{T:tma1} and \ref{T:tma2}.\end{cor}

\begin{proof}
Note first that if $\f$ is bounded then
\begin{eqnarray*}
(\om+dd^c\chi\circ\f)^n & = &
\left[(1-\chi'(\f))\om+\chi'(\f)\,\om_{\f}\right]^n+ \\ &&
n\chi''(\f)\,d\f\wedge d^c\f\wedge\left[(1-\chi'(\f))\om+
\chi'(\f)\,\om_{\f}\right]^{n-1},\end{eqnarray*} which equals the
measure from (\ref{e:atten}). In the general case, let
$\f^s=\max(\f,s)$, $s<0$. As $s\rightarrow-\infty$ we have, by the
proofs of Theorems \ref{T:tma1} and \ref{T:tma2}, that
\begin{eqnarray*}&&
\left(\e_\ell'(\f^s)\,d\f^s\wedge d^c\f^s\wedge
\om_{\f^s}^{n-\ell-1}\wedge\om^\ell\right)(\{\f\leq
s\})\rightarrow0,\\&& \left(\eta_\ell(\f^s)\,\om_{\f^s}^\ell\wedge
\om^{n-\ell}\right) (\{\f\leq s\})\rightarrow0.\end{eqnarray*}
Since $\max(\chi(\f),-j)=\chi(\f^{s_j})$, where
$s_j=\chi^{-1}(-j)\rightarrow-\infty$ as $j\rightarrow+\infty$, we
conclude by formula (\ref{e:atten}) applied to $\f^{s_j}$ that
$$\left(\om+dd^c\max(\chi(\f),-j)\right)^n
(\{\chi(\f)\leq-j\})\rightarrow0,$$ so $\chi \circ \f \in
{\mathcal E}(X,\om)$. Moreover, using the continuity of the
operators in Theorems \ref{T:tma1} and \ref{T:tma2}, it follows
that (\ref{e:atten}) holds for $\f$.\end{proof}

\begin{exas}
\text{ }

1) For $\chi(t)=-(-t)^p$, Corollary \ref{C:atten} shows that
$-(-\f)^p \in {\mathcal E}(X,\om)\subset DMA(X,\om)$ for all
$p<1$, although it usually does not have gradient in $L^2(X)$ when
$p\geq1/2$. Functions with less attenuated singularities can be
obtained by considering for instance $\chi(t)=t/\log(M-t)$.

2) Recall from Example \ref{E:Evans} that there are functions $\f
\notin DMA(X,\om)$ such that the current $\om_{\f}$ has support in
$\{\f=-\infty\}$. The smoothing effect of the composition with
$\chi$ is quite striking: indeed, by Corollary \ref{C:atten} and
Proposition \ref{P:wgrad},
$$
(\om+dd^c \chi \circ \f)^n=[1-\chi'(\f)]^n \om^n+
n[1-\chi'(\f)]^{n-1}\chi''(\f)\, d\f \wedge d^c \f\wedge\om^{n-1}
$$
is a measure with density in $L^1(X,\om^n)$.
\end{exas}

\section{The local vs. global domains of definition}\label{S:local}

Cegrell found in \cite{Ce04} the largest class of psh functions on
a bounded hyperconvex domain on which the Monge-Amp\`ere operator
is well defined, stable under maximum and continuous under
decreasing limits. Later, Blocki gave in \cite{Bl06} a complete
characterization of the domain of definition of the Monge-Amp\`ere
operator on {\em any} open set in $\C^n$, $n \geq 2$. For an open
subset $ U \subset \C^n$, the domain of definition $\mathcal D (U)
\subset PSH (U)$ of the Monge-Amp\`ere operator on $U$ is given by
$(n - 1)$ local boundedness conditions on weighted gradients
\cite{Bl06}. In particular, for $n = 2$, $\mathcal D (U) = PSH (U)
\cap W_{loc}^{1,2} (U)$, where $W_{loc}^{1,2}(U)$ is the Sobolev
space of functions in $L^2_{loc}(U)$ with locally square
integrable gradient \cite{Bl04}.

\par We  describe here the class of $\om$-psh
functions on $X$ which locally belong to the domain of definition
of the Monge-Amp\`ere operator. As we shall see, it is smaller
than the global domain $DMA(X,\om)$.

\begin{defi}\label{D:DMAloc} Let $DMA_{loc} (X,\om)$ be the
set of functions $\f\in PSH (X,\om)$ such that locally, on any
small open coordinate chart $U \subset X$, the psh function $\f|_U
+ \rho_U \in \mathcal D (U)$, where $\rho_U$ is a psh potential of
$\om$ on $U$.
\end{defi}

The goal of this section is to establish the following:

\begin{thm}\label{T:DMAloc} We have
$$
DMA_{loc}(X,\om)=\bigcap_{1 \leq p \leq n-1} {\mathcal E}^p(\om^p,\om).
$$
Moreover
$$
DMA_{loc}(X,\om) \subset  {\mathcal E}^1 (\om,\om) \subset
\widehat{DMA}_{\leq n}(X,\om).
$$
\end{thm}

Here ${\mathcal E}^p(\om^p,\om)$ denotes the class ${\mathcal
E}_{\chi}(T,\om)$ for $T=\om^p$ and $\chi(t)=-(-t)^p$.

\begin{proof}
We first show that $DMA_{loc}(X,\om)= \bigcap_{1 \leq p \leq n-1}
{\mathcal E}^p(\om^p,\om)$. By \cite{Bl06}, a $\om$-psh function
$\f\in DMA_{loc}(X,\om)$ if and only if
\begin{equation}\label{e:5}
\sup_j \int_X (-\f_j)^{p-1}d\f_j
\wedge d^c \f_j  \wedge \om_{\f_j}^{n - p - 1} \wedge \om^{p} < +
\infty, \; 1\le p\le n - 1,
\end{equation}
where $\{\f_j\}$ is any sequence of bounded $\om$-psh functions
decreasing to $\f$ on $X$. For our first claim we need to prove
that (\ref{e:5}) is equivalent to the following
\begin{equation}\label{e:3star}
\sup_j \int_X (-\f_j)^{p} \om_{\f_j}^{n - p} \wedge \om^{p} < +
\infty, \; 1\le p\le n - 1,
\end{equation}
 This follows by integration by parts. Indeed for $\psi \in PSH (X,\om)
\cap L^{\infty} (X)$, $\psi \leq - 1$ and $1 \leq p \leq n - 1$, we have
\begin{eqnarray*}
\int_{X} (- \psi)^{p}  \om_{\psi}^{n - p} \wedge \om^{p} & = &
\int_X (- \psi)^{p}  \om_{\psi}^{n - p - 1} \wedge \om^{ p + 1} \\
& + &  p \int_{X} (- \psi)^{p - 1}  d \psi \wedge d^c \psi  \wedge
\om_{\psi}^{n - p - 1} \wedge \om^{p}.
\end{eqnarray*}
By iterating the above formula we get
\begin{eqnarray*}
 \int_{X} (- \psi)^{ p}  \om_{\psi}^{n - p} \wedge \om^{ p}&=&
 \int_X (- \psi)^{p} \om^n \\
&+& p \sum_{k=0}^{n - p - 1} \int_{X} (- \psi)^{ p - 1} d \psi
\wedge d^c \psi \wedge \om_{\psi}^{n - p - k - 1} \wedge \om^{ p +
k}.
\end{eqnarray*}
This yields $(\ref{e:5}) \Longleftrightarrow (\ref{e:3star})$, so
$DMA_{loc}(X,\om)=\bigcap_{1 \leq p \leq n-1} {\mathcal
E}^p(\om^p,\om)$.

By Corollary \ref{C:econvergence} we have ${\mathcal
E}^1(\om^p,\om) \subset \widehat{DMA}_{n - p + 1} (X,\om)$, for
any $1 \leq p \leq n - 1$. Therefore  the class
$$
{\mathcal E}^1(\om,\om):=
\left\{ \f \in PSH(X,\om) \, / \, \sup_j \int_X |\f_j| (\om+dd^c \f_j)^{n-1} \wedge \om
<+\infty \right\}
$$
is contained in $\widehat{DMA}_{\leq n}(X,\om)$ (here
$\f_j:=\max(\f,-j)$ denote as usually the canonical approximants),
since

$$
{\mathcal E}^1(\om,\om)=
\bigcap_{1 \leq p \leq n-1} {\mathcal E}^1(\om^p,\om).
$$
This equality also shows that
$DMA_{loc} (X,\om) \subset {\mathcal E}^1(\om,\om)$.

\end{proof}

\begin{rqe}
A class similar to ${\mathcal E}^1(\om,\om)$ has been very
recently considered by Y. Xing in \cite{X}.
\end{rqe}

Note that in dimension $n=2$, the class ${\mathcal
E}^1(\om,\om)=DMA_{loc}(X,\om)$ is simply the set of $\om$-psh
function whose gradient is in $L^2(X)$. However when $n \geq 3$,
the class ${\mathcal E}^1(\om,\om)$ is strictly larger than
$DMA_{loc}(X,\om)$, as the following example shows:

\begin{exa}
Observe that ${\mathcal E}^1(X,\om) \subset {\mathcal E}^1(\om,\om)$.
We are going to exhibit an example of a function
$\f$ such that $\f \in {\mathcal E}^1(X,\om)$, but
$\f \notin L^2(\om_{\f} \wedge \om^{n-1})$.
This will show that $\f \notin DMA_{loc}(X,\om)$ when $n \geq 3$.

Assume $X = \mathbb P^{n - 1} \times \mathbb P^1$ and $\om (x,y)
:= \a (x) + \b (y),$ where $\a$ is the Fubini-Study form on
$\mathbb  P^{n - 1}$ and $\b$ is the Fubini-Study form on $\mathbb
P^1$. Fix $u \in PSH (\mathbb  P^{n - 1},\a) \cap \mathcal
C^{\infty} (\mathbb  P^{n - 1})$ and  $v \in  \mathcal E (\mathbb
P^1,\b)$. The function $\f$ defined by  $\f (x,y) := u (x) + v
(y)$ for $(x,y) \in X$ belongs to $\mathcal E (X,\om)$. Moreover
$\om_{\f} = \a_u + \b_v$ and for any $1 \leq \ell \leq n$, we have
$$\om_{\f}^{\ell} = \a_u^\ell +\ell \a_u^{\ell - 1} \wedge \b_v.$$
Therefore
$$
\f \in L^p (\om_{\f}^n) \Longleftrightarrow v \in L^p (\b_{v})
\Longleftrightarrow \f \in L^p (\om_{\f}^{\ell} \wedge \om^{n - \ell}).
$$
Thus choosing  $v \in L^1 (\om_v) \setminus L^{2} (\om_{v})$,
we obtain an example of a $\om$-psh function $\f$
such that
$\f \in \mathcal E^1 (X,\om) \subset \mathcal E^1 (\om,\om)$ but $\f \notin DMA_{loc} (X,\om).$
\end{exa}

We finally observe that there are functions in $DMA_{loc}(X,\om)$
which do not belong to the class ${\mathcal E}(X,\om)$, since the
latter cannot have positive Lelong numbers.

\begin{lem}\label{L:ELel}
Let $\f \in PSH(X,\om)$ be such that $\f \geq c \log
dist(\cdot,p)$, for some $c>0$ and $p \in X$. Then $\f \in
DMA_{loc}(X,\om)$, and
$$
\f \in {\mathcal E}(X,\om) \text{ if and only if }
\nu(\f,p)=0.
$$
\end{lem}

\begin{proof}
If $\r_p \in PSH(X,\om)$ is a function comparable to $c \log
dist(\cdot,p)$, then by Proposition \ref{P:div}, $\f, \r_p \in
DMA_{loc}(X,\om)$. We can assume without loss of generality that $\r_p
\leq \f \leq 0$. Note that the positive Radon measure $\om_{\f}
\wedge \om_{\r_p}^{n-1}$ is well defined on $X$ and has a Dirac
mass at $p$ if and only if $\nu(\f,p)>0$ (see \cite{De93}).

\par It follows from the proof of Proposition \ref{P:cLn} that if
$\nu(\f,p)>0$, then $\om_{\f}^n$ has a Dirac mass at $p$, hence
$\f \notin{\mathcal E}(X,\om)$. If $\nu(\f,p)=0$, then we can find
a convex increasing function $\chi:\R^- \rightarrow \R^-$ such
that $\chi(-\infty)=-\infty$ and $\int_X (-\chi) \circ \r_p \,
\om_{\f} \wedge \om_{\r_p}^{n-1} <+\infty$. Using Stokes theorem
(in the spirit of the fundamental inequality in \cite{GZ2}), it
follows that
$$
\int_X (-\chi) \circ \f \, \om_{\f}^n \leq 2^{n-1}
\int_X (-\chi) \circ \r_p \, \om_{\f} \wedge \om_{\r_p}^{n-1} <+\infty,
$$
hence $\f \in {\mathcal E}_{\chi}(X,\om) \subset {\mathcal E}(X,\om)$.
\end{proof}

\section{Sobolev classes}\label{S:Sob}

\subsection{Weighted gradients}\label{SS:Wgrad}

Let $\chi:\R\rightarrow \R$ be a convex increasing function of
class $C^2$. If $\f \in PSH(X,\om) $ is smooth then
\begin{equation}\label{e:deriv}
\om+dd^c \chi \circ \f=\chi'' \circ \f \, d\f \wedge d^c \f +\chi'
\circ \f \, \om_{\f}+(1-\chi' \circ \f)\om.
\end{equation}
So if $\chi'(-1) \leq 1$ and $\f \leq -1$, then $\chi \circ \f \in
PSH(X,\om)$. It is well-known that ($\om$-)psh functions have
gradient in $L^{2-\e}_{loc}$ for all $\e>0$ \cite{Ho}, but in
general not in $L^2_{loc}$. The previous computation indicates
that a weighted version of the gradient is in $L^2(X)$.

\par We denote by $W^{1,2}(X,\om)$ the set of functions $\f \in PSH(X,\om)$
whose gradient is square integrable. Since $\omega$ is
K\"ahler, we can in fact define, for $\f\in PSH(X,\omega)$, the
function $|\nabla\f|=|\nabla\f|_\omega$ a.e. on $X$ by
$$|\nabla\f|^2:=d\f\wedge d^c\f\wedge\omega^{n-1}/\omega^n.$$
Note that $\f\in W^{1,2}(X,\omega)$ if and only if $|\nabla\f|\in
L^2(X,\omega^n)$.

\begin{pro}\label{P:wgrad}
Let $\chi:\R \rightarrow \R$ be a convex increasing function of
class ${\mathcal C}^2$. Then for every $\f
\in PSH(X,\om)$,
$$
\int_X \chi''\circ\f\,|\nabla
\f|^2\,\om^{n}\leq\sup_X\chi'(\f)\int_X\omega^n.
$$
In particular, if $\f\leq-1$ and $0<p<1/2$, then $-(-\f)^p \in
W^{1,2}(X,\om)$.
\end{pro}

\begin{proof} Let $M=\sup_X\chi'(\f)$ and $\f^j :=
\max(\f, - j)$ for $j \in \N$. It follows from (\ref{e:deriv})
that
$$
\int_X \chi'' \circ\f^j\,|\nabla \f^j|^2\,\om^{n} =\int_X \chi''
\circ\f^j\,d\f^j\wedge d^c\f^j \wedge \om^{n-1}\leq M\int_X\om^n.
$$
This shows that the sequence $f_j :=\chi''(\varphi^j)\,|\nabla
\varphi^j|^2$ is bounded in $L^1(X,\omega^n)$. Since
$\f_j\searrow\f$, we use \cite[Theorem 4.1.8]{Ho} to conclude,
after taking a subsequence, that $f_j\to\chi''(\varphi)\,|\nabla
\varphi|^2$ a.e. on $X$. The inequality in the statement now
follows from Fatou's lemma.

\par For our second claim, set $\f_a:=-(-\f)^a$, $0<a<1$.
Applying (\ref{e:deriv}) with $\chi(t)=-(-t)^a$ yields $\f_a\in
PSH(X,\omega)$ and $$\int_X (-\f)^{a-2} d\f \wedge d^c \f \wedge
\om^{n-1} \leq \frac{1}{a(1-a)}\,\int_X\omega^n.$$ If $p=a/2$ then
$d \f_p \wedge d^c \f_p=p^2 (-\f)^{2p-2} d\f \wedge d^c \f$, hence
$$
\int_X d\f_p\wedge d^c\f_p\wedge \om^{n-1}\leq
\frac{p}{2(1-2p)}\,\int_X\omega^n.
$$
\end{proof}

\par Note that using  $\chi(t)=t/\log(M-t)$, for $M$ large enough,
one can improve the previous weighted $L^2$-bound to $$ \int_X
(-\f)^{-1}[\log(M-\f)]^{-2} |\nabla \f|^2\,\om^n <+\infty. $$
Observe also that this cannot be improved much farther: in the
case when $\om_{\f}$ is the current of integration along a
hypersurface,  $-(-\f)^{1/2}$ does not have square integrable
gradient.

Next, we show that functions from the class
$DMA_{loc}(X,\om)$ satisfy
stronger weighted gradient boundedness conditions.

\begin{pro}\label{P:WnW12} If
$\f \in DMA_{loc} (X,\om)$, $\f\leq-1$, then
$$ \int_X (- \f)^{n - 2} d \f \wedge d^c \f \wedge
\om^{n - 1} \leq \frac{1 }{n - 1} \int_X (- \f)^{n - 1} \om_{\f}
\wedge\om^{n-1}.$$ In particular, $DMA_{loc}(X,\om)\subset
W^{1,2}(X,\om)$.
\end{pro}

\begin{proof} If $\f_j :=\max(\f,-j)$, $j
\in \N$, then integrating by parts we get
$$
(n-1)\int_X(-\f_j)^{n-2}d\f_j\wedge d^c\f_j\wedge\om^{n - 1}
\leq\int_X (- \f_j)^{n - 1} \om_{\f_j} \wedge\om^{n-1}.
$$
The sequence of functions $f_j:=(-\f_j)^{n - 2}|\nabla \f_j|^2$
is thus uniformly bounded in
$L^1(X,\om^n)$. Since $\f_j \searrow \f$, the conclusion follows
as in the previous proof by \cite[Theorem 4.1.8]{Ho} and by
Fatou's lemma.
\end{proof}

\subsection{Blowing up and down}\label{SS:blowup}
We saw in Section \ref{S:E} that there are many functions $\f \in
DMA(X,\om)$ whose gradient does not belong to $L^2(X)$. This
condition, although the best possible in the local two-dimensional
theory, is thus not the right one from the global point of view.
We show here that this condition does not behave well under a
birational change of coordinates. For simplicity, and without loss
of generality, we restrict ourselves to the two-dimensional local
setting.

\par Let $\pi:Y\rightarrow B$ be the blow up at the origin of
a ball $B \subset \C^2$, and let $E$ be the exceptional divisor.

\begin{lem}\label{L:intblow}
If $\delta>0$ then
$\pi^\star\left(L^{2+\delta}_{loc}(B)\right)\subset L^1_{loc}(Y)$,
but $\pi^\star\left(L^2_{loc}(B)\right)\not\subset L^1_{loc}(Y)$.
\end{lem}

\begin{proof} Let $z=(x,y)\in B$, and fix a coordinate chart in $Y$
given by $(s,t)\rightarrow(s,st,[1:t])\in Y$. Then
$\pi(s,t)=(x,y)$, $x=s$, $y=st$. For fixed positive constants $M$
and $a$, let
$$\Delta=\{(s,t):\,|s|<a,\;|t|<M\}\;,\;
K=\pi(\Delta)=\{(x,y):\,|x|<a,\;|y|<M|x|\}.$$

\par Let $f\in L^{2+\delta}_{loc}(B)$, and $\widetilde
f=\pi^\star f$, so $\widetilde f(s,t)=f(s,st)$. Note that the
function $z\rightarrow|x|^{-2}$ is in $L^{2-\epsilon}(K)$ for
every $\epsilon>0$. We take $\gamma=\delta/(1+\delta)$ and apply
H\"older's inequality with conjugate exponents $2+\delta$ and
$2-\gamma$:
$$\int_\Delta|\widetilde f|=\int_K\frac{|f|}{|x|^2}
\leq\left(\int_K|f|^{2+\delta}\right)^{1/(2+\delta)}
\left(\int_K\frac{1}{|x|^{4-2\gamma}}\right)^{1/(2-\gamma)}.$$
Hence $\widetilde f\in L^1_{loc}(Y)$. The function
$f(z)=1/(\|z\|^2\log\|z\|)$ is in $L^2_{loc}(B)$, but $|\widetilde
f(s,t)|\geq C|s|^{-2}(|\log|s||+1)^{-1}$ on $\Delta$, for some
$C>0$. So $\widetilde f\not\in L^1({\Delta})$.
\end{proof}

\par In dimension $n$, a similar proof shows that
$\pi^\star\left(L^{n+\delta}_{loc}(B)\right)\subset L^1_{loc}(Y)$.

\begin{exa}\label{E:w12blow}
$\pi^\star\left(PSH\cap W^{1,
4}_{loc}(B)\right)\not\subset W^{1,2}_{loc}(Y)$. Indeed, let
$$u_\alpha(z)=-(-\log\|z\|)^\alpha\;,\;0<\alpha<1.$$ One checks
easily that $u_\alpha\in W^{1,4}(B)$ if $\alpha<3/4$. Let
$\widetilde u_\alpha=\pi^\star u_\alpha$. With the notation from
the proof of Lemma \ref{L:intblow}, we have for $|s|$ small
$$\widetilde
u_\alpha(s,t)=-\left(-\log|s|-\log\sqrt{1+|t|^2}\right)^\alpha
\;,\;\left|\frac{\partial\widetilde u_\alpha}{\partial
s}(s,t)\right|\geq \frac{C}{|s|(-\log|s|)^{1-\alpha}}\;,$$ for
some constant $C>0$. So $\widetilde u_\alpha\not\in
W^{1,2}_{loc}(\Delta)$ if $\alpha\geq1/2$. Note that if $u\in
PSH\cap W^{1, 4+\delta}_{loc}(B)$, $\delta>0$, then, by the
Sobolev embedding theorem, $u$ is continuous and so is $\pi^\star
u$. Hence $\pi^\star u\in W^{1,2}_{loc}(Y)$ (see \cite{Bl04}).
\end{exa}

\par The previous example shows that the condition $\nabla \f \in
L^2(X)$ does not behave well under blow-up. We show it behaves
well under blowing down.

\begin{pro}\label{P:blowdown}
Let $T$ be a positive closed current of bidegree (1,1) on $B$, and
let $R=\pi^\star T-\nu[E]$, where $\nu=\nu(T,0)$. If $R$ has psh
potentials in $W^{1,2}_{loc}(Y)$, then $T$ has psh potentials in
$W^{1,2}_{loc}(B)$.
\end{pro}

Recall that any positive closed $(1,1)$-current in the blow up
of $B$ at the origin writes $R=\pi^*T-\nu [E]+\l[E]$, where
$T$ is a positive closed current in $B$, $\nu=\nu(T,0)$ is the Lelong
number of $T$ at the origin, and $\l \geq 0$.
Clearly $R$ cannot have potentials in $W_{loc}^{1,2}$ if $\l>0$.

\begin{proof} Let $u$ be a psh potential of $T$ on $B$.
We only have to check that the gradient of $u$ is $L^2$-integrable
in a neighborhood of the origin. Using the notation from the proof
of Lemma \ref{L:intblow}, a psh potential for $R$ on $\Delta$ is
$v(s,t)=u(s,st)-\nu\log|s|$. Hence on $K$ we have
$u(x,y)=v(x,y/x)+\nu\log|x|$. Therefore for almost all $(x,y)\in
K$
$$\frac{\partial u}{\partial x}(x,y)=\frac{\partial
v}{\partial s}(x,y/x)-\frac{y}{x^2}\,\frac{\partial v}{\partial
t}(x,y/x)+\frac{\nu}{2x}\;,\;\frac{\partial u}{\partial
y}(s,t)=\frac{1}{x}\,\frac{\partial v}{\partial t}(x,y/x).$$ Now
\begin{eqnarray*}\int_K\frac{|y|^2}{|x|^4}\left|\frac{\partial v}
{\partial t}(x,y/x)\right|^2&=&\int_\Delta|t|^2\left|
\frac{\partial v}{\partial t}(s,t)\right|^2,\\
\int_K\frac{1}{|x|^2}\left|\frac{\partial v}{\partial
t}(x,y/x)\right|^2&=&\int_\Delta\left|\frac{\partial v}{\partial
t}(s,t)\right|^2.\end{eqnarray*} Since the function
$(x,y)\rightarrow1/x$ is in $L^2(K)$, we conclude that the partial
derivatives of $u$ belong to $L^2(K)$. A similar argument shows
that the partial derivatives of $u$ are in $L^2(K')$, where
$K'=\{(x,y):\,|y|<a,\;|x|<M|y|\}$.
\end{proof}

\subsection{Compact singularities}\label{SS:compactsing}
We give here an important class of functions in $DMA_{loc}(X,\om)$. Let
$D$ be a divisor on $X$ and set
$$L_D^{\infty}(X,\om)=\{ \f \in PSH(X,\om) \,/\, \f \text{ is bounded
near } D \}.$$ Thus the singularities of $\f \in
L_D^{\infty}(X,\om)$ are constrained to a compact subset of $X
\setminus D$. When $D=H$ is a hyperplane of the complex projective
space $X=\P^n$, the set $L_D^{\infty}(X,\om)$ is in one-to-one
correspondence with the Lelong class ${\mathcal L}^+(\C^n)$ of psh
functions $u$ in $\C^n$ such that $u(z)-\log||z||$ is bounded near
infinity. So these are the $\om$-psh analogues of the psh
functions with compact singularities introduced by Sibony in
\cite{Si85} (see also \cite{De93}).

\begin{pro}\label{P:div} If $D$ is an ample divisor then
$L_D^{\infty}(X,\om)\subset DMA_{loc}(X,\om)$.
\end{pro}

\begin{proof} Fix $\f \in L_D^{\infty}(X,\om)$ and let $V$ be a small
neighborhood of $D$, so that $\f$ is bounded in $V$. We can assume
that $\f \leq 0$ and $\int_X\om^n=1$.

\par Let $\om'$ be a smooth semi-positive closed $(1,1)$ form in the
cohomology class of $D$, such that $\om' \equiv 0$ in $X \setminus
V$. Since $D$ is ample, $\om'$ is cohomologous to a K\"ahler form
$\omega_0$. For simplicity, we assume $\om_0=\om$ (otherwise we
bound $\om\leq C\om_0$ in all arguments below). Hence
$\om=\om'+dd^c \chi$, where $\chi$ is a smooth function on $X$,
chosen to be either negative or positive, as we like.

\par We assume here $\chi \geq 0$, and we first observe that $\f \in
L^1(\om_{\f} \wedge \om^{n-1})$:
\begin{eqnarray*}
\int_X (-\f)&& \! \! \! \! \! \! \! \! \! \! \! \! \! \! \!
\om_{\f} \wedge \om^{n-1} =
\int_X (-\f) \om_{\f} \wedge \om'\wedge \om^{n-2}+
\int_X (-\f) \om_{\f} \wedge dd^c \chi \wedge \om^{n-2} \\
&\leq& ||\f||_{L^{\infty}(V)} \int_X \om_{\f} \wedge \om' \wedge \om^{n-2}
+\int_X \chi \om_{\f} \wedge (-dd^c \f ) \wedge \om^{n-2}\\
&\leq& ||\f||_{L^{\infty}(V)}+||\chi||_{L^{\infty}(X)}<+\infty,
\end{eqnarray*}
since $-dd^c \f \leq \om$, $\chi \om_{\f} \geq 0$ and
$\int_X \om_{\f} \wedge \om' \wedge \om^{n-2}=\int_X \om^n=1$.

\par It follows that the positive current $\om_{\f}^2:=\om_{\f}
\wedge \om+dd^c(\f \om_{\f})$ is well defined. We can thus show by
a similar argument that $\f \in L^1(\om_{\f}^2 \wedge \om^{n-2})$,
so that $\om_{\f}^3$ is also well defined, and so on. At last, we
show that $\f \in L^1(\om_{\f}^{n-1} \wedge \om)$.

\par We now prove that $\f^2 \in L^1(\om_{\f}^{n-2} \wedge \om^2)$.
We assume here that $\chi \leq 0$. Observe that $ -dd^c \f^2=-2
d\f \wedge d^c \f -2\f dd^c \f \leq 2(-\f) \om_{\f}, $ therefore
$$
\int_X\f^2  \om_{\f}^{n-2} \wedge \om^2\leq \int_X \f^2
\om_{\f}^{n-2} \wedge \om' \wedge \om+ 2\|\chi\|_{L^{\infty}(X)}
\int_X (-\f) \om_{\f}^{n-1} \wedge \om.
$$
The integrals are finite because $\om'$ has support in $V$, where
$\f$ is bounded, and because $\f \in L^1(\om_{\f}^{n-1} \wedge
\om)$. Thus $\f^2 \in L^1(\om_{\f}^{n-2} \wedge \om^2)$. Similar
integration by parts allows us to show that $\f^3 \in
L^1(\om_{\f}^{n-3} \wedge \om^3)$, by using $\f^2 \in
L^1(\om_{\f}^{n-2} \wedge \om^2)$. Continuing like this, we see
that $\f \in DMA_{loc}(X,\om)$.
\end{proof}

\par Proposition \ref{P:div} shows that $PSH(X,\om) \cap
L^{\infty}_{loc}(X \setminus \{p\}) \subset DMA_{loc}(X,\om)$. Indeed,
if $X$ is {\it projective}, one can find for each $p \in X$ a
divisor $D \not\ni p$. In the general case, it follows from the
proof that one only needs to construct a smooth semi-positive form
$\om'$ cohomologous to a K\"ahler form, and such that $\omega'
\equiv 0$ near $p$. This can be achieved on any {\it K\"ahler}
manifold.

\section{Concluding remarks}\label{S:conclusion}

In this section we restrict our attention to the two-dimensional case. In the local setting of an
open subset $U\subset{\Bbb C}^2$, a psh function $u$ on $U$ belongs to the domain of definition
${\mathcal D}(U)$ if and only if the gradient of $u$ is locally square integrable ($W^{1,2}$) on
 $U$ \cite{Bl04}. Such characterization allows to prove important properties of
${\mathcal D}(U)$, such as convexity and stability under taking the maximum of elements of
${\mathcal D}(U)$ with arbitrary psh functions (see \cite{Bl04}).

\par Let $X$ be a compact K\"ahler surface and $\omega$ be a K\"ahler form on $X$. In order to prove
further properties of $DMA(X,\omega)$ it would be useful to obtain equivalent characterizations for
this domain. We present here the connection in certain cases between  $DMA(X,\omega)$ and certain
energy classes.

\subsection{Direct sums}\label{SS:P1P1}
Let $X={\Bbb P}^1\times{\Bbb P}^1$
and $\omega=\omega_1+\omega_2$, where
$\omega_i=\pi_i^*\omega'$ is the pull-back of the
Fubini-Study form $\om'$ of ${\Bbb P}^1$ by the projection onto the $i^{th}$ factor, $i=1,2$.

\begin{pro}\label{P:energyP1} Let $\varphi\in PSH(X,\omega)$ be of the form
$$
\varphi(x,y)=u(x)+v(y),
\text{ where }
u,v\in PSH({\Bbb P}^1,\omega').
$$
Then
\par (i) $\varphi\in{\mathcal E}^1(\omega,\omega)\Longleftrightarrow
u,v\in{\mathcal E}^1({\Bbb P}^1,\omega')\Longleftrightarrow \varphi\in{\mathcal E}^1(X,\omega)$;

\par(ii) $\varphi\in {\mathcal E}(\omega,\omega)\Longleftrightarrow
u,v\in{\mathcal E}({\Bbb P}^1,\omega')\Longleftrightarrow\varphi\in{\mathcal E}(X,\omega)$;

\par (iii) $\f \in DMA(X,\om) \Longleftrightarrow \f \in {\mathcal E}(\omega,\omega)$.
\end{pro}

\begin{proof} $(i)$ Note that
\begin{eqnarray*}
\int\varphi\,\omega_\varphi\wedge\omega& = &
\int u\,\omega_{1,u}\wedge\omega_2+\int v\,\omega_1\wedge\omega_{2,v}+ \\
&&\int u\,\omega_1\wedge\omega_{2,v}+
\int v\,\omega_{1,u}\wedge\omega_2,
\end{eqnarray*}
where the integrals $\int u\,\omega_1\wedge\omega_{2,v}$ and
$\int v\,\omega_{1,u}\wedge\omega_2$ are always finite by Fubini's theorem.
We use here the obvious notations $\om_{1,u}:=(\om_1+dd^c u)(x)$
and $\om_{2,v}:=(\om_2+dd^c v)(y)$.
This shows that
$\varphi\in{\mathcal E}^1(\omega,\omega)$ if and only if $u,v\in{\mathcal E}^1({\Bbb P}^1,\omega')$.

\par If $u,v\in{\mathcal E}^1({\Bbb P}^1,\omega')$ then $\varphi\in W^{1,2}$, hence
$\varphi\in\widehat{DMA}(X,\omega)$ and $\omega_\varphi^2=2\omega_{1,u}\wedge\omega_{2,v}$.
By Fubini's theorem
$$\int\varphi\,\omega_\varphi^2=2\int_{{\Bbb P}^1}u\,\omega_{1,u}+2\int_{{\Bbb P}^1}v\,\omega_{2,v},$$ so
$\omega_\varphi^2(\{\varphi=-\infty\})=0$. We conclude that $\varphi\in{\mathcal E}(X,\omega)$, hence
$\varphi\in{\mathcal E}^1(X,\omega)$. This formula also shows that $\varphi\in{\mathcal E}^1(X,\omega)$
implies that $u,v\in{\mathcal E}^1({\Bbb P}^1,\omega')$.
\vskip.2cm

\par $(ii)$ {\em and} $(iii)$.
The equivalence  $\varphi\in{\mathcal E}(\omega,\omega)\Longleftrightarrow
u,v\in{\mathcal E}({\Bbb P}^1,\omega')$ is a direct consequence of the following equality:
$$\int_{\{\varphi=-\infty\}}\omega_\varphi\wedge\omega=
\int_{\{u=-\infty\}}\omega_{1,u}+\int_{\{v=-\infty\}}\omega_{2,v}.$$

\par We show next that $u,v\in{\mathcal E}({\Bbb P}^1,\omega')\Longrightarrow
\varphi\in{\mathcal E}(X,\omega)$. Let $\varphi_j,u_j,v_j$ be the canonical approximants
and set $E_j=\{u>-j\}\cap\{v>-j\}$. Since the bounded $\omega$-psh functions
$\varphi_{2j}$ and $u_j+v_j$ coincide
on the plurifine open set $E_j$, we have by \cite{BT87} that
$${\bf 1}_{E_j}\omega_{\varphi_{2j}}^2=2\cdot{\bf 1}_{E_j}\omega_{u_j}\wedge\omega_{v_j}=
2\cdot{\bf 1}_{\{u>-j\}}\omega_{u_j}\wedge{\bf 1}_{\{v>-j\}}\omega_{v_j}.$$ Note that the product measure
$\omega_u\wedge\omega_v$ puts full mass 2 on the set $\{u>-\infty\}\cap\{v>-\infty\}$, and that
${\bf 1}_{\{\varphi>-2j\}}\omega_{\varphi_{2j}}^2\geq{\bf 1}_{E_j}\omega_{\varphi_{2j}}^2$. This shows
that the sequence of measures ${\bf 1}_{\{\varphi>-2j\}}\omega_{\varphi_{2j}}^2$ increases to a measure
with total mass 2, hence $\varphi\in{\mathcal E}(X,\omega)$.

\par We conclude the proof by showing that $\varphi\in DMA(X,\omega)$ implies that
$u,v\in{\mathcal E}({\Bbb P}^1,\omega')$. Assume for a contradiction that
$v\not\in{\mathcal E}({\Bbb P}^1,\omega')$. We may assume that there exists a compact
$K\subset\{v=-\infty\}\cap\{[1:w]:\,w\in{\Bbb C}\}$ so that
$$
\int_Kdd^c V=a>0\;,\;\;V(w):=\log\sqrt{1+|w|^2}+v([1:w]).
$$
We use here (bi)homogeneous coordinates $[z_0:z_1], [w_0:w_1]$ on
$X$. Let $u_j$ be smooth $\omega'$-psh functions decreasing to $u$
on ${\Bbb P}^1$, and set
\begin{eqnarray*}
U_j(z)&=&\log\sqrt{1+|z|^2}+u_j([1:z]),\\
\Phi_j(z,w)&=&\max(U_j(z)+V(w),\log|z-\zeta|-j),\end{eqnarray*}
where $\zeta\in{\Bbb C}$. The functions $\Phi_j$ yield functions
$\varphi_j\in DMA(X,\omega)$ decreasing to $\varphi$, hence
$(dd^c\Phi_j)^2\to\omega_\varphi^2$ on ${\Bbb C}^2\subset X$. We
will show that
$$\int_{\{z=\zeta\}\times K}\omega_\varphi^2\geq a.$$ Since $\zeta$ is arbitrary, we get a contradiction.

\par For $r>0$ let $\chi_1\geq0$ be a smooth function such that $\chi_1(z)=1$ in the closed disc $E_r$ of
radius $r$ centered at $\zeta$ and $\chi_1$ is supported in the disc $D_r$ of radius $2r$ centered at
$\zeta$. For fixed $j$ let $N$ be an open neighborhood of $K$ so that
$V(w)<\log r-j-\max_{D_r}U_j$ for $w\in N$, and let $\chi_2\geq0$ be a smooth function supported in $N$
such that $\chi_2(w)=1$ on $K$. Let $\chi(z,w)=\chi_1(z)\chi_2(w)$. Since $dd^c\chi\wedge dd^c\Phi_j$ is
supported on the open set $\{U_j(z)+V(w)<\log|z-\zeta|-j\}$ it follows that
\begin{eqnarray*}
\int\chi\,(dd^c\Phi_j)^2&=&\int(U_j+V)\,dd^c\chi\wedge dd^c\Phi_j\\
&\geq&\int\chi\,dd^cV\wedge dd^c\Phi_j=\int\Phi_j\chi_2\,dd^c\chi_1(z)\wedge dd^cV(w).
\end{eqnarray*}
Note that
$$
U_j(z)+V(w)<\log r-j<\log|z-\zeta|-j
$$
on the support of $\chi_2dd^c\chi_1$, thus
\begin{eqnarray*}
\int\chi\,(dd^c\Phi_j)^2&\geq&\int(\log|z-\zeta|-j)\chi_2\,dd^c\chi_1\wedge dd^cV\\
&=&\int\chi\,dd^c\log|z-\zeta|\wedge dd^cV\geq\int_K dd^cV=a.
\end{eqnarray*}
We conclude that $$\int_{E_r\times K}\omega_\varphi^2\geq
\limsup_{j\to\infty}\int_{E_r\times K}(dd^c\Phi_j)^2\geq a,$$ and as $r\rightarrow0$, that
$\int_{\{z=\zeta\}\times K}\omega_\varphi^2\geq a.$
\end{proof}

\subsection{The case $X={\Bbb P}^2$.}\label{SS:P2}

We produce now similar examples in the case of $X=\P^2$ with
$\om$ the Fubini-Study form.
Let $[t:z:w]$ denote the homogeneous coordinates and $\f$ be a $\om$-psh
function with Lelong number 1 at point $p=[1:0:0]$.
It is easy to see that $\f$ can be written as
\begin{equation} \label{e:P2example}
\f[t:z:w]=\frac{1}{2} \log \frac{|z|^2+|w|^2}{|t|^2+|z|^2+|w|^2} +u[z:w]
\end{equation}
where $u$ is a $\om'$-psh function on $\{t=0\}\simeq \P^1$.

\begin{pro}
If $u \notin {\mathcal E}(\P^1,\om)$ then $\f \notin DMA(\P^2,\om)$.
\end{pro}

\begin{proof}
Suppose $\f \in DMA(\P^2,\om)$ and let $\f_j$ be functions defined
by (\ref{e:P2example}) with $u$ replaced by $u_j$, where $u_j$ are
bounded $\om$-psh on $\P^1$ decreasing to $u$. Then $\f_j$
decreases to $\f$ and $\om_{\f_j}^2=\d_p$ is the Dirac mass at
$p$, hence $\om_\f^2=\d_p$.

On the other other hand, we are going to construct another sequence of
functions $\p_j \in DMA(\P^2,\om)$ decreasing to $\f$ such that
$\om_{\p_j}^2$ does not converge to $\d_p$.
Let $K\subset\{t=0\}$ be a compact so that $u=-\infty$ on $K$ and $\om_u(K)=a>0$, and let
$$
\p_j([t:z:w])=\max\left(\f([t:z:w]),\log|t|-\frac{1}{2}\,\log(|t|^2+|z|^2+|w|^2)-j\right).
$$
Then $\p_j\in DMA_{loc}(\P^2,\om)$ by Proposition \ref{P:div}. One
can show, as in the proof of Proposition \ref{P:energyP1}, that
$\om_{\p_j}^2(K)\geq a$, for all $j$. This contradicts that
$\om_\f^2=\d_p$.
\end{proof}

For functions $\f$ as in (\ref{e:P2example}), it is easy to show
that $\f \in DMA_{loc}(\P^2,\om)$ if and only if $u \in
W^{1,2}(\P^1)$. It is an interesting question whether $\f \in
DMA(\P^2,\om)$ if $u \in {\mathcal E}(\P^1,\om')$. A concrete
example is
$$
\f_{\a}[t:z:w]:=\frac{1}{2} \log \frac{|z|^2+|w|^2}{|t|^2+|z|^2+|w|^2}
-\left[ 1- \frac{1}{2} \log \frac{|z|^2}{|z|^2+|w|^2} \right]^{\a},
$$
where $0<\a<1$.
Then $\f_{\a} \notin {\mathcal E}(\P^2,\om)$ since it has positive Lelong number
at $p$, and $\f_{\a} \notin DMA_{loc}(\P^2,\om)$ if $\a \geq 1/2$.
It would be of interest to know if $\f_{\a} \in DMA(\P^2,\om)$
for some $\a \in [1/2,1]$.

\subsection{A candidate?}

Previous examples indicate that in dimension $n=2$, the class
$$
{\mathcal E}(\om,\om):=\left\{ \f \in PSH(X,\om) \, / \, (\om+dd^c
\f)(\{\f=-\infty\})=0 \right\}
$$
plays a central role. Note that it enjoys several interesting properties:
\begin{itemize}
\item ${\mathcal E}(\om,\om)$ is convex and stable under maximum;
\item ${\mathcal E}(\om,\om)={\mathcal E}(\om',\om') \cap PSH(X,\om)$ whenever $\om \leq \om'$;
\item $DMA_{loc}(X,\om) \subset {\mathcal E}^1(\om,\om) \subset {\mathcal E}(\om,\om)$;
\item ${\mathcal E}(X,\om) \subset {\mathcal E}(\om,\om)$.
\end{itemize}

Together with the special examples analyzed in sections
\ref{SS:P1P1} and \ref{SS:P2}, this motivates the following:

\begin{ques}
Assume $n=\dim_{\C} X=2$.

\noindent Do we have $DMA(X,\om) \subset {\mathcal E}(\om,\om)$
and/or ${\mathcal E}(\om,\om) \subset DMA(X,\om)$?
\end{ques}


\begin{thebibliography}{XXXX}

\bibitem[BT1]{BT76} E. BEDFORD and B. A. TAYLOR: The Dirichlet problem
for a complex Monge-Amp\`ere equation. Invent. Math. {\bf 37}
(1976), no. 1, 1--44.

\bibitem[BT2]{BT78} E. BEDFORD and B. A. TAYLOR: Variational properties
of the complex Monge-Amp\`ere equation. I. Dirichlet principle.
Duke Math. J. {\bf 45} (1978), no. 2, 375--403.

\bibitem[BT3]{BT82} E. BEDFORD and B. A. TAYLOR: A new capacity for
plurisubharmonic functions. Acta Math. {\bf 149} (1982), no. 1-2,
1--40.

\bibitem[BT4]{BT87} E. BEDFORD and B. A. TAYLOR: Fine topology,
\v{S}ilov boundary, and $(dd^c)^n$. J. Func. Anal. {\bf 72}
(1987), no. 2, 225--251.

\bibitem[Bl1]{Bl04} Z. BLOCKI: On the definition of the Monge-Amp\`ere
operator in $\C^2$. Math. Ann. {\bf 328} (2004), no 3, 415-423.

\bibitem[Bl2]{Bl06} Z. BLOCKI: The domain of definition of the complex
Monge-Amp\`ere operator. Amer. J. Math. {\bf 128} (2006), no. 2,
519--530.

\bibitem[BK]{BK07} Z. BLOCKI and S. KOLODZIEJ: On regularization
of plurisubharmonic functions on manifolds. Proc. Amer. Math. Soc.
{\bf 135} (2007), 2089--2093.

\bibitem[C1]{Ce98} U. CEGRELL: Pluricomplex energy. Acta Math. {\bf 180}
(1998), no. 2, 187--217.

\bibitem[C2]{Ce04} U. CEGRELL: The general definition of the complex
Monge-Amp\`ere operator. Ann. Inst. Fourier (Grenoble)  {\bf 54}
(2004),  no. 1, 159--179.

\bibitem[D1]{De92} J.-P. DEMAILLY: Regularization of closed positive
currents and intersection theory. J. Algebraic Geom.  {\bf 1}
(1992),  no. 3, 361--409.

\bibitem[D2]{De93} J.-P. DEMAILLY: Monge-Amp\`ere operators, Lelong
numbers and intersection theory. Complex analysis and geometry,
115--193, Univ. Ser. Math., Plenum, New York (1993).

\bibitem[EGZ]{EGZ} P. EYSSIDIEUX, V. GUEDJ and A. ZERIAHI: Singular
K\"ahler-Einstein metrics. Preprint arXiv math.AG/0603431.

\bibitem[FS]{FS95} J. E. FORN\AE SS and N. SIBONY: Oka's
inequality for currents and  applications. Math. Ann. {\bf 301}
(1995), no. 3, 399-419.

\bibitem[GZ1]{GZ05} V. GUEDJ and A. ZERIAHI: Intrinsic capacities on
compact K\"ahler manifolds. J. Geom. Anal. {\bf 15} (2005), no. 4,
607-639.

\bibitem[GZ2]{GZ2} V. GUEDJ and A. ZERIAHI: The weighted Monge-Amp\`ere
energy of quasiplurisubharmonic functions. J. Func. Anal. (2007), to
appear.

\bibitem[H\"o]{Ho} L. H\"ORMANDER: Notions of convexity. Progress in
Mathematics, 127. Birkh\"auser Boston, Inc., Boston, MA, (1994)
viii+414 pp.

\bibitem[K1]{Ko98} S. KOLODZIEJ: The complex Monge-Amp\`ere equation.
Acta Math. {\bf 180} (1998), no. 1, 69--117.

\bibitem[K2]{Ko03} S. KOLODZIEJ: The Monge-Amp\`ere equation on compact
K\"ahler manifolds. Indiana Univ. Math. J.  {\bf 52}  (2003),  no.
3, 667--686.

\bibitem[Sib]{Si85} N. SIBONY: Quelques probl\`emes de prolongement de
courants en analyse complexe. Duke Math. J.  {\bf 52}  (1985),
no. 1, 157--197.

\bibitem[Siu]{Siu74} Y. T. SIU: Analyticity of sets associated to
Lelong numbers and the extension of closed positive currents.
Invent. Math. {\bf 27} (1974), 53--156.

\bibitem[T]{Ti} G. TIAN: Canonical metrics in K\"ahler geometry.
Lectures in Mathematics ETH Z\"urich. Birkh\"auser Verlag, Basel
(2000).

\bibitem[X]{X} Y.XING: The general definition of the complex
Monge-Amp\`ere operator on compact K\"ahler manifolds. Preprint
arXiv:0705.2099.

\bibitem[Y]{Yau} S. T. YAU: On the Ricci curvature of a compact K\"ahler
manifold and the complex Monge-Amp\`ere equation. I. Comm. Pure
Appl. Math. {\bf 31} (1978), no. 3, 339--411.

\end{thebibliography}
\end{document}